\documentclass[12pt,reqno]{amsart}
\usepackage{amsmath, amsthm, amsopn, amssymb, microtype}
\usepackage{mathrsfs} 
\usepackage{mathscinet}
\usepackage{bbm}
\usepackage{enumerate, tikz, etoolbox, intcalc, geometry, caption, subcaption, afterpage}
\usepackage[english]{babel}
\usepackage{enumitem}
\usepackage[toc,page]{appendix}

\usepackage[colorlinks=true,urlcolor=blue,pdfborder={0 0 0}]{hyperref}
\hypersetup{linkcolor=[rgb]{0,0,0.6}}
\hypersetup{citecolor=[rgb]{0,0.6,0}}
\usepackage[nameinlink]{cleveref}
\crefname{theorem}{Theorem}{Theorems}
\crefname{thm}{Theorem}{Theorems}
\crefname{mainthm}{Theorem}{Theorems}
\crefname{conj}{Conjecture}{Theorems}
\crefname{lemma}{Lemma}{Lemmas}
\crefname{lem}{Lemma}{Lemmas}
\crefname{remark}{Remark}{Remarks}
\crefname{prop}{Proposition}{Propositions}
\crefname{defn}{Definition}{Definitions}
\crefname{corollary}{Corollary}{Corollaries}
\crefname{cor}{Corollary}{Corollaries}
\crefname{section}{Section}{Sections}
\crefname{figure}{Figure}{Figures}
\crefname{quest}{Question}{Questions}

\newcommand{\N}{\mathbb{N}}
\newcommand{\Z}{\mathbb{Z}}

\newcommand{\R}{\mathbb{R}}
\newcommand{\B}{\mathcal{B}}

\newcommand{\CC}{\mathsf{C}}
\newcommand{\T}{\mathbb{T}}

\newcommand{\bR}{\mathbf{R}}

\graphicspath{ {./Pictures/} }

\newcommand{\ignore}[1]{}

\newtheorem{thm}{Theorem}[section]
\newtheorem{lemma}[thm]{Lemma}
\newtheorem{prop}[thm]{Proposition}
\newtheorem{cor}[thm]{Corollary}
\newtheorem{claim}[thm]{Claim}

\theoremstyle{definition}

\newtheorem{remark}[thm]{Remark}

\newtheorem*{thm*}{Theorem}
\newtheorem*{prop*}{Proposition}
\newtheorem*{Prob*}{Problem}

\newif\ifdraft\drafttrue

\begin{document}

\title[Multidimensional LCLT in deterministic systems]{Multidimensional local limit theorem in deterministic systems and an application to non-convergence of polynomial multiple averages}

\author{Zemer Kosloff}
\address{Einstein Institute of Mathematics, Hebrew University of Jerusalem, Edmond J. Safra Campus, Jerusalem, 9190401, Israel.}
\email{zemer.kosloff@mail.huji.ac.il}

\author{Shrey Sanadhya}
\address{Einstein Institute of Mathematics, Hebrew University of Jerusalem, Edmond J. Safra Campus, Jerusalem, 9190401, Israel.}
\email{shrey.sanadhya@mail.huji.ac.il}


\begin{abstract}
We show that for every ergodic and aperiodic probability preserving system $(X,\B,m,T)$, there exists $f:X\to \mathbb{Z}^d$, whose corresponding cocycle satisfies the $d$-dimensional local central limit theorem. 

We use the $2$-dimensional result to resolve a question of Huang, Shao and Ye and Franzikinakis and Host regarding non-convergence in $L^2$ of polynomial multiple averages of non-commuting zero entropy transformations. Our methods also give the first examples of failure of multiple recurrence for zero entropy transformations along polynomial iterates. 
\end{abstract}

\maketitle

\section{Introduction}

Given a probability preserving transformation $(X,\B,m,T)$, and a function $g:X\to\mathbb{R}$, its \textbf{sum process} is defined by $S_n(g):=\sum_{k=0}^{n-1}g\circ T^k$, $n\in\mathbb{N}$. Similarly if $f:X\to\mathbb{Z}^d$ is a $d$-dimensional map given by $f (x) = (f^{(1)} (x),\ldots, f^{(d)} (x)$ for $x \in X$, then $f\circ T (x)= (f^{(1)}\circ T(x),\ldots, f^{(d)}\circ T(x))$ and the sum process represents component wise summation. In other words for $n\in\mathbb{N}$,
\[
S_n(f):=\sum_{k=0}^{n-1}f\circ T^k = \sum_{k=0}^{n-1} (f^{(1)}\circ T^k,\ldots, f^{(d)}\circ T^k).
\]
We will also refer to $S_n(f)$ as the cocycle corresponding to $f$. 

In \cite{MR891642}, Burton and Denker proved the following surprising result: For every $(X,\B,m,T)$ an aperiodic and probability preserving system there exists a square integrable function whose corresponding cocycle satisfies a non-degenerate central limit theorem. In \cite{MR1624218},  Voln\'y proved the existence of a function $f$ whose corresponding cocycle  satisfies the central limit theorem and its corresponding sum process converges to a non-degenerate Brownian motion. In particular the variance of $S_n(f)$ grows linearly. For a comprehensive history of such results we refer the reader to the introduction of \cite{MR1624218} and \cite{MR4374685}.

In \cite{MR4374685}, it was shown by the first author and Voln\'y that for every ergodic and aperiodic probability preserving system $(X,\B,m,T)$, there exists a $\Z$ valued function whose corresponding cocycle satisfies a lattice local central limit theorem. In the first part of this work, we prove the following $d$-dimensional lattice local central limit theorem. 

\begin{thm}\label{thm:LCLT}
Let $(X,\B,m,T)$ be an ergodic and aperiodic probability preserving transformation. There exists a square integrable function $f:X\to\mathbb{Z}^d$ with $\int_X fdm = 0$, such that
\[
\sup_{x\in\Z^d} n^{d/2}\left|m\left(S_n(f)=x\right)-\frac{1}{(2\pi n\sigma^2)^{d/2}}e^{-\frac{\|x\|^2}{2n\sigma^2}}\right|\xrightarrow[n\to\infty]{}0,
\]
where $\sigma^2=2 (\ln 2)^2$.
\end{thm}
The local limit theorem of \cite{MR4374685}, is the $d=1$ case of \cref{thm:LCLT}. 

Given $S,T$, two measure preserving transformations of the probability space $(X,\mathcal{B},m)$ and $f,g\in L^2(X,m)$ the corresponding double averages are defined by
\[
\mathbb{A}_n(f,g):=\frac{1}{N}\sum_{n=0}^{N-1}f\circ T^{n}g\circ S^{n}. 
\]
When $S$ and $T$ commute, Conze and Lesigne \cite{MR788966} proved that for all $f,g\in L^\infty(m)$, $\mathbb{A}_n(f,g)$ converges in $L^2$ as $n\to\infty$. This was extended by Bergelson and Leibman \cite{Bergelson_Leibman_2002} to the case where $T$ and $S$ generate a nilpotent group. Given $p_1$ and $p_2$, two integer polynomials one can ask whether the corresponding polynomial double averages 
\begin{equation}\label{eq:polymial_averages}
\frac{1}{N}\sum_{n=0}^{N-1}f\circ T^{p_1(n)} g\circ S^{p_2(n)} 
\end{equation}
converge in $L^2$ for every $f,g\in L^\infty(X,m)$. Walsh \cite{MR2912715} showed that \eqref{eq:polymial_averages} converges in $L^2$ when $T$ and $S$ generate a nilpotent group. When $S$ and $T$ have positive entropy, the limit in \eqref{eq:polymial_averages} may not exist for certain bounded functions, see for example \cite[Proposition 1.4]{MR4585298}.

The following $L^2$ convergence result without commutativity was proved by Frantzikinakis and Host. 
\begin{thm*}\cite[Theorem 1.1]{MR4585298}
Let $T,S$ be measure preserving transformations acting on a probability space $(X,\B,\mu)$ such that the system $(X,\B,\mu,T)$ has zero entropy. Let also $p\in\Z[t]$, an integer polynomial of degree greater or equal to $2$. Then for every $f,g\in L^\infty(m)$, the limit 
\[
\lim_{N\to\infty}\frac{1}{N}\sum_{n=0}^{N-1}f\circ T^n g\circ S^{p(n)}
\]
exists in $L^2(\mu)$.
\end{thm*}
It was further showed in \cite[Proposition 1.4]{MR4585298} that the assumption that $T$ has zero entropy is essential. The following problem was posed by Frantzikinakis and Host.  

\begin{Prob*} \cite{MR4585298} Let $T,S$ be measure preserving transformations acting on a probability space $(X,\B,\mu)$ such that the system $(X,\B,\mu,T)$ has zero entropy. Does the $L^2$ convergence result holds when in place of the iterates $n,p(n)$ we use the pair of iterates $n,n$ or $n^2,n^3$? 

In general, does the result hold for pairs of iterates given by arbitrary polynomials $p_1,p_2 \in \Z[t]$ with $p_1(0)=p_2(0)=0$ and $\deg(p_1),\deg(p_2)\geq 2$?
\end{Prob*}
Huang, Shao and Ye proved the following result. 

\begin{thm*}\cite{Huang_Shao_Ye_2024} 
Let $p_1,p_2:\Z\to\Z$ be polynomials with $\deg (p_1),\deg (p_2)\geq 5$. For any $F\subset \N$ and $c\in \left(0,\frac{1}{2}\right)$, there exist $T,S$, two ergodic, measure preserving transformations of a standard probability space $(X,\mathcal{X},\mu)$ with $h_\mu(X,T)=h_\mu(X,S)=0$ and two measurable subsets $A_1,A_2\in\mathcal{X}$, and $M \in \N$ such that for all $n \geq M$
\[
\mu\left(A_1\cap T^{-p_1(n)}A_2\cap S^{-p_2(n)}A_2\right)=\begin{cases}
 0, & \text{if}\ n\in F,\\
 c, & \text{if}\ n\notin F.
\end{cases}
\]
As a consequence, there exist $T,S$ two ergodic, zero-entropy systems of a standard probability space $(X,\mathcal{X},\mu)$ and $A_2\in\mathcal{X}$ with $\mu(A_2)>0$ such that the averages 
\[
\frac{1}{N}\sum_{n=0}^{N-1} 1_{A_2}\circ T^{p_1(n)}1_{A_2}\circ S^{p_2(n)}
\]
do not converge in $L^2(\mu)$. 
\end{thm*}
Huang, Shao and Ye conjecture (see \cite[Conjecture 1.2]{Huang_Shao_Ye_2024}) that for any two integer polynomials $p_1,p_2$ with $p_1(0)=p_2(0)=0$, the non-convergence result holds unless there exists an integer $c\neq 0$ such that $p_1(n)=cn$ and $p_2$ is a polynomial of degree $2$ or higher. Using the $2$-dimensional local limit theorem in \cref{thm:LCLT} we show the following.
 
\begin{thm}\label{thm:Main}
Let $p_1,p_2:\Z\to\Z$ be polynomials with $\deg (p_1),\deg (p_2) \geq 2$. There exist $T,S$ two ergodic measure preserving transformations of a standard probability space $(X,\mathcal{X},\mu)$, with $h_{\mu}(X,T) = h_{\mu}(X,S) = 0$ and $A\in\mathcal{X}$ with $\mu(A)>0$, such that the averages 
\[
\frac{1}{N}\sum_{n=0}^{N-1} 1_{A}\circ T^{p_1(n)}1_{A}\circ S^{p_2(n)}
\]
do not converge in $L^2(\mu)$. 
\end{thm}

Recently the non-convergence result for the iterates $p_1(n)=p_2(n)=n$ was proved (independently) by Austin \cite{2024arXiv240708630A}, Huang Shao and Ye \cite{2024arXiv240710728H} and Rhyzikov \cite{2024arXiv240713741R}. Austin remarks that his methods give a non-convergence result for iterates of the type $p(n)$, $p(n)$ for $p$ a general integer polynomial of degree at least $1$. As Austin's examples are Gaussian systems, the following is a simple consequence of \cite{2024arXiv240708630A}, its proof is given in \Cref{Appendix C}. 

\begin{prop}\label{prop: Austin}
For every $d,c\in\Z\setminus \{0\}$ there exists $T,S$, two ergodic measure preserving transformations of a standard probability space $(X,\mathcal{X},\mu)$, with $h_{\mu}(X,T) = h_{\mu}(X,S) = 0$ and $A\in\mathcal{X}$ with $\mu(A)>0$, such that the averages
\[
\frac{1}{N}\sum_{n=0}^{N-1} 1_{A}\circ T^{cn}1_{A}\circ S^{dn}
\]
do not converge in $L^2(\mu)$. 
\end{prop}

The combination of \cref{thm:Main} and \cref{prop: Austin} gives a full solution to the question of Frantzikinakis and Host and the conjecture of Huang, Shao and Ye.

Finally, for some special cases of polynomials, our methods give the following counterexamples for the recurrence problem for zero entropy transformations. 

\begin{thm}\label{thm:recurrence counterexample}
There exists $M>0$, such that $\N\ni L>M$ and $d\geq 3$, there exists $T,S$, two ergodic measure preserving transformations of a standard probability space $(X,\mathcal{X},\mu)$, with $h_{\mu}(X,T) = h_{\mu}(X,S) = 0$ and $A\in\mathcal{X}$ with $\mu(A)>0$, such that for all $n\in\N$,
\[
\mu\left(A\cap T^{-Ln^d}A\cap S^{-Ln^d}A\right)=0.
\]
\end{thm}

\subsubsection{Notation}
Here and throughout $\log(x)$ denotes the logarithm of $x$ to base $2$ while $\ln(x)$ is the natural logarithm of $x$. 

Given $x\in\R^d$. $\|x\|:=\sqrt{\sum_{j=1}^dx_j^2}$ is its Euclidean norm and $\|x\|_\infty=\max_{1\leq j\leq d}|x_j|$ is its supremum norm. Now when $(X,\B,m)$ is a probability space and $F:X\to \R^d$, we write $\|F\|_2=\sqrt{\int_X \|F(x)\|^2dm(x)}$.

{\bf Acknowledgement.} The authors would like to thank Wen Huang, Song Shao, and Xiangdong Ye for communicating their results,  Nikos Frantzikinakis for asking us to investigate the recurrence problem, and Dalibor Voln\'y for remarks on early versions. This work was partially supported by the Israel Science Foundation grant No. 1180/22. 

\section{Construction of the function in Theorem \ref{thm:LCLT}}\label{sec:construction}

Let $2\leq D\in\mathbb{N}$ and $(X,\B,m,T)$ be an ergodic and aperiodic probability preserving system. Denote by $U:L^2(X,\mu)\to L^2(X,\mu)$ the corresponding Koopman operator of $T$. Slightly more generally, if $F$ is a function from $X$ to $\R^D$ then $UF=F\circ T$.

The function in our proof of \cref{thm:LCLT} is of the form $F=(f^{(1)},\ldots,f^{(D)})$ where for each $1\leq i\leq D$, $f^{(i)}$ is a sum of couboundaries and satisfies the local limit theorem result in dimension $1$ of \cite{MR4374685}. A priori, this would not guarantee that $F$ satisfies the multi-dimensional local limit theorem. The key point is that, we carefully construct the coboundaires in the functions together to guarantee that $S_n(F)$ can be expressed as the sum
\[
S_n(F)=\big(Y_1(n),\ldots,Y_D(n)\big)+Z(n),
\]
where $Y_1(n),\dots,Y_D(n)$ are independent and satisfy the local limit theorem in dimension-$1$ and $Z(n)$ is independent from $(Y_1(n),\ldots,Y_D(n))$ with a small $L^2$ norm.

For $k \in \N$ we set,
\[
p_k:=\begin{cases}
2^k, & k\ \text{even},\\
2^k+1, &\ k\ \text{odd},
\end{cases}
\] and
$d_k:=2^{k^2}$. Similarly, let $\alpha_1=\frac{1}{2}$ and $\alpha_k:=\frac{1}{p_k\sqrt{k\log(k)}}$ for $k\geq 2$. 

\begin{prop}\label{prop:functions_defn}
Let $(X,\B,m,T)$ be an ergodic and aperiodic probability preserving system and $D\in\mathbb{N}$. There exists $\bar{f}_k^{(i)}:X\to \{-1,0,1\}$, $k\in\mathbb{N},i\in\{1,\ldots,D\}$ such that:
\begin{enumerate}[label=(\alph*)]
\item\label{prop_sec_a:functions_defn} For all $k\in\mathbb{N}$ and $i\in\{1,\dots, D\}$, 
\[
\mu\left(\bar{f}_k^{(i)}=1\right)=\mu\left(\bar{f}_k^{(i)}=-1\right)=\frac{\alpha_k^2}{2}.
\]
\item\label{prop_sec_b:functions_defn} For every $k\in\N$, the functions $\left\{\bar{f}_k^{(i)}\circ T^j:\ 0\leq j\leq 2d_k+p_k,\ 1\leq i\leq D\right\}$ are i.i.d. 
\item\label{prop_sec_c:functions_defn} For every $k\in\N$, the functions $\left\{\bar{f}_k^{(i)}\circ T^j:\ 0\leq j\leq 2d_k+p_k,\ 1\leq i\leq D\right\}$ are independent of 
\[
\mathcal{A}_k=\left\{\bar{f}_l^{(i)}\circ T^j:\ 1\leq l<k,1\leq i\leq D,\ 0\leq j\leq 2d_k+p_k\right\}. 
 \]
\end{enumerate}
\end{prop}
We say that a partition $\xi$ of $X$ is a \textbf{measurable partition} if all the atoms of $\xi$ are measurable. A function $g:X\to \R$ is \textbf{independent of the measurable partition} $\xi$, if for all $a<b$ and $\alpha\in\xi$
\[
m\left(\alpha\cap [a\leq g<b]\right)=m(\alpha)m(a\leq g<b). 
\]\cref{prop:functions_defn} is a simple consequence of the following result from \cite{MR4374685}. 
\begin{prop}\cite[Proposition 2]{MR4374685}\label{prop:copying_lemma}
Let $(X,\B,m,T)$ be an ergodic and aperiodic probability preserving system and $\xi$ a measurable partition of $X$. Given a finite set $A$ and $X_1,\ldots,X_m$ a collection of $A$ valued i.i.d. random variables, there exists $g:X\to A$ such that $\left(g\circ T^j\right)_{j=0}^{m-1}$ is independent of $\xi$ and distributed as $\left(X_j\right)_{j=1}^m$. 
\end{prop}
\begin{proof}[Proof of \cref{prop:functions_defn}]
We construct the functions by induction on $k$. First let $\xi=\left\{\emptyset,X\right\}$ be the trivial partition. Applying \cref{prop:copying_lemma}, we can find $g_1:X\to\{-1,0,1\}$ such that $\{g_1\circ T^j\}_{j=0}^{D(2d_1+p_1)}$ are i.i.d. and 
\[
m\big(g_1=1\big)=m\big(g_1=-1\big)=\frac{\alpha_1^2}{2}. 
\] 
Setting for $J\in\{1,\ldots,D\}$, $\bar{f}_1^{(J)}=U^{(J-1)(2d_1+p_1)}g_1$, the base of construction is complete. Suppose we have chosen $\bar{f}_k^{(i)}:X\to \{-1,0,1\}$, $k\leq N,i\in\{1,\ldots,D\}$ such that:

\begin{itemize}
\item[($a^\prime$)] For all $k\leq N$ and $i\in\{1,\dots, D\}$, 
\[
\mu\left(\bar{f}_k^{(i)}=1\right)=\mu\left(\bar{f}_k^{(i)}=-1\right)=\frac{\alpha_k^2}{2}.
\]
\item[($b^\prime$)] For every $k\leq N$, the functions $\left\{\bar{f}_k^{(i)}\circ T^j:\ 0\leq j\leq 2d_k+p_k,\ 1\leq i\leq D\right\}$ are i.i.d. 
\item[($c^\prime$)] For every $k\leq N$, the functions $\left\{\bar{f}_k^{(i)}\circ T^j:\ 0\leq j\leq 2d_k+p_k,\ 1\leq i\leq D\right\}$ are independent of 
\[
\mathcal{A}_k=\left\{\bar{f}_l^{(i)}\circ T^j:\ 1\leq l<k,1\leq i\leq D,\ 0\leq j\leq 2d_k+p_k\right\}. 
\]
\end{itemize} Let $\xi_{N+1}$ be the finite measurable partition of $X$ according to the values of the (finite valued) vector function
\[
\left(\bar{f}_l^{(i)}\circ T^j:\ 1\leq l\leq N,1\leq i\leq D,\ 0\leq j\leq 2d_{N+1}+p_{N+1}\right).
\]
Applying \cref{prop:copying_lemma} again, we obtain a function $g_{N+1}:X\to\{-1,0,1\}$ such that $\xi_{N+1}$ is independent of the $\{-1,0,1\}$ valued i.i.d. sequence $\left\{g_{N+1}\circ T^j:0\leq j<D\left(2d_{N+1}+p_{N+1}\right)\right\}$ and 
\[
m\left(g_{N+1}=1\right)=m\left(g_{N+1}=-1\right)=\frac{(\alpha_{N+1})^2}{2}.
\]
Define for all $J\in\{1,\ldots,D\}$, $\bar{f}_{N+1}^{(J)}:=U^{(J-1)(2d_{N+1}+p_{N+1})}g_{N+1}$. We check that \ref{prop_sec_a:functions_defn}, \ref{prop_sec_b:functions_defn} and \ref{prop_sec_c:functions_defn} hold. Indeed, \ref{prop_sec_a:functions_defn} holds since for all $J\in\{1,\ldots,D\}$, $\bar{f}_{N+1}^{(J)}$ and $g_{N+1}$ are equally distributed. As
\[
\left\{\bar{f}_{N+1}^{(i)}\circ T^j:\ 0\leq j\leq 2d_{N+1}+p_{N+1},\ 1\leq i\leq D\right\}=\left\{g_{N+1}\circ T^j:0\leq j<D\left(2d_{N+1}+p_{N+1}\right)\right\},
\] part \ref{prop_sec_b:functions_defn}, follows for $N+1$. Finally noting that being independent of $\xi_{N+1}$ is equivalent to being independent of $\mathcal{A}_{N+1}$, part \ref{prop_sec_c:functions_defn} follows for $k=N+1$. 
\end{proof}
From now on let $\bar{f}_k^{(i)}:X\to \{-1,0,1\}$, $k\in\mathbb{N},i\in\{1,\ldots,D\}$ be the functions from \cref{prop:functions_defn}. For each $i\in \{1,\ldots, D\}$, let $f^{(i)}=\sum_{k=1}^\infty f_k^{(i)}$, where for $k\in \mathbb{N}$
\begin{equation}\label{eq:def_function_f_k_i}
    f_k^{(i)}:=\sum_{j=0}^{p_k-1}U^j\bar{f}_k^{(i)}-U^{d_k}\left(\sum_{j=0}^{p_k-1}U^j\bar{f}_k^{(i)}\right). 
\end{equation}
\subsection{The local limit function in one dimension}\label{subsec:llt dim1}
In \cite{MR4374685}, \cref{prop:copying_lemma} was applied to construct a sequence of functions $\bar{f}_k:X\to\{-1,0,1\}$ so that 
\begin{enumerate}[label=(\alph*)]
\item\label{item_a:bar_f} For all $k\in\mathbb{N}$, 
\[
\mu\left(\bar{f}_k=1\right)=\mu\left(\bar{f}_k=-1\right)=\frac{\alpha_k^2}{2}.
\]
\item\label{item_b:bar_f} For every $k\in N$, the functions $\left\{\bar{f}_k\circ T^j:\ 0\leq j\leq 2d_k+p_k\right\}$ are i.i.d. 
\item\label{item_c:bar_f} For every $k\in N$, the functions $\left\{\bar{f}_k\circ T^j:\ 0\leq j\leq 2d_k+p_k\right\}$ are independent of 
\[
\mathcal{A}_k=\left\{\bar{f}_l\circ T^j:\ 1\leq l<k, \ 0\leq j\leq 2d_k+p_k\right\}. 
\]
\end{enumerate} After this the first author and Voln\'y defined the function $f=\sum_{k=1}^\infty f_k$, where for each $k \in \N$,
\[
f_k:=\sum_{j=0}^{p_k-1}U^j\bar{f}_k-U^{d_k}\left(\sum_{j=0}^{p_k-1}U^j\bar{f}_k\right).
\]
\begin{thm*}
If $\bar{f}_k:X\to\{-1,0,1\}$ satisfies \ref{item_a:bar_f}, \ref{item_b:bar_f} and \ref{item_c:bar_f}, and $f_k$ and $f$ are as above, then
\begin{itemize}
    \item \cite[Proposition 3]{MR4374685} $f \in L^2 (X,\mu)$. 
    \item \cite[Theorem 4]{MR4374685} $f$ satisfies the local limit theorem with $\sigma^2 := 2(\ln 2)^2 $. That is
\[
\sup_{x\in\Z} \left| \sqrt{n} \mu (S_n(f) = x) -\frac{e^{-x^2/(2 n \sigma^2)}}{\sqrt{2\pi \sigma^2}} \right|\xrightarrow[n\to\infty]{}0.
\]
\end{itemize}
\end{thm*} Fix $i\in\{1,\ldots,D\}$. The functions $\left(f_k^{(i)}\right)_{k=1}^\infty$ is distributed as $(f_k)_{k=1}^\infty$. In addition, $\bar{f}_k^{(i)}:X\to\{-1,0,1\}$ satisfies \ref{item_a:bar_f}, \ref{item_b:bar_f} and \ref{item_c:bar_f}. \cref{cor:f_i}, follows from the above theorem.  
\begin{cor}\label{cor:f_i} For each $i \in \{1,\ldots, D\}$,

\begin{itemize}
    \item[(a)] $f^{(i)} \in L^2 (X,\mu)$. 
    \item[(b)] $f^{(i)}$ satisfies the local limit theorem with $\sigma^2 := 2(\ln 2)^2 $. That is
\[
\sup_{x\in\Z^d} \left| \sqrt{n} \mu (S_n\left(f^{(i)}\right) = x) -\frac{e^{-x^2/(2 n \sigma^2)}}{\sqrt{2\pi \sigma^2}} \right|\xrightarrow[n\to\infty]{}0.
\]
\end{itemize}
\end{cor} 
\vspace{3.5mm}
For $k\in\N$, we write 
\begin{equation}\label{eq:def_function_F_k}
F_k:=\left(f_k^{(1)},\ldots,f_k^{(D)}\right),
\end{equation} where $f_k^{(i)}$ are as in \eqref{eq:def_function_f_k_i}, for $1\leq i \leq D$. Set $F:X\to\Z^d$ to be the function
\begin{equation}\label{eq:def_function_F}
F:=\sum_{k=1}^\infty F_k=\left(f^{(1)},\ldots,f^{(D)}\right).
\end{equation}
We will show, using the extra independence we introduced in the construction and the arguments in \cite{MR4374685}, that $F$ satisfies the multi-dimensional lattice local central limit theorem (\cref{thm:LCLT}). For convenience, we recall the statement of the Theorem below. 

\begin{thm}\label{thm:LCLT_1} The function $F:X\to\mathbb{Z}^D$, satisfies the local limit theorem with $\sigma^2 := 2(\mathrm{ln}\, 2)^2 $. That is
\[
\sup_{x\in\Z^D} n^{D/2}\left|m\left(S_n(F)=x\right)-\frac{1}{(2\pi n\sigma^2)^{D/2}}e^{-\frac{\|x\|^2}{2n\sigma^2}}\right|\xrightarrow[n\to\infty]{}0.
\]
\end{thm}

\section{Proof of Theorem \ref{thm:LCLT_1}}
The proof follows similar lines as \cite{MR4374685} and uses the estimates that appear there. The idea is to use the arguments in \cite{MR4374685} and the extra independence we introduced in the construction of the functions above. First, We will describe the strategy of proof in \cite{MR4374685} and the results we will exploit. On the way, we will correct a minor error in \cite{MR4374685}

\subsection{Review of the arguments and results on the local limit theorem for $d=1$} Assume $\bar{f}_k:X \rightarrow \Z$, $k\in \N$ is a sequence of functions satisfying \ref{item_a:bar_f}, \ref{item_b:bar_f} and \ref{item_c:bar_f} in \cref{subsec:llt dim1} and let $f_k$ and $f=\sum_{k=1}^\infty f_k$ be the functions defined there. 

The proof of the $1$-dimensional local limit theorem for $S_n(f)$  is done by decomposing 
\[
S_n(f)=Z_{SM}(f)(n)+\hat{Y}_f(n)+Z_{La}(f)(n),
\]
where 
\begin{align*}
Z_{SM}(f)(n)&:= \sum_{\{k : d_k\leq n\}}S_n\left(f_k\right)\\
\hat{Y}_f(n)&:= \sum_{\{k : p_k <n < d_k\}}S_n\left(f_k\right)\\
Z_{La}(f)(n)&:= \sum_{\{k: n \leq p_k\}}S_n\left(f_k\right). 
\end{align*}

\begin{lemma}\label{lem:first_bound_in_KV}\cite[Lemma 7]{MR4374685}
For every $n\in\mathbb{N}$, the random variables $Z_{Sm}(f)(n),\hat{Y}_f(n),Z_{La}(f)(n)$ are independent and 
\[
\left\|Z_{Sm}(f)(n)+Z_{La}(f)(n)\right\|_2^2=O\left(\frac{n}{\sqrt{\log(n)}}\right). 
\]
\end{lemma} Essentially, with the aid of \cref{prop:bla} in \Cref{App_1}, this lemma says that if one can prove the local central limit theorem for $\hat{Y}_f(n)$, then one would obtain it for $S_n(f)$. The next stage is to look at the main contributing terms in $\hat{Y}_f(n)$. To that end, for $n \in \N$ and $j\in\{1,\ldots,n\}$, set
\[
Y_j(f,n):=\sum_{\{k: p_k\leq j+1,\ p_k<n<d_k\}} p_k \left(U^j\bar{f}_k-U^{d_k+j}\bar{f}_k\right),
\]
and 
\[
\mathscr{W}_f(n):=\sum_{j=1}^{n-1} Y_j(f,n).
\]
\begin{prop}
\label{prop:second_bound_in_KV}\cite[Proposition 9]{MR4374685} For every $n\in\mathbb{N}$, the random variables $\mathscr{W}_f(n)$ and $S_n(f)-\mathscr{W}_f(n)$ are independent and
\[
\left\|S_n(f)-\mathscr{W}_f(n)\right\|_2^2=O\left(\frac{n}{\sqrt{\log(n)}}\right). 
\]
\end{prop} We note that writing $B_k:=p_k\sum_{j=p_k-1}^{n-1}U^j\bar{f}_k$, we have
\[
\mathscr{W}_f(n)=\sum_{\{k:\ p_k<n<d_k\}}\left(B_k-U^{d_k}B_k\right).
\]
Now we set,
\[
\hat{I}_n:=\{k\in 2\N:\ p_{k+1}<n<d_k\}.
\]
For $k\in \hat{I}_n$, let
\[
V_k:=\sum_{i=2^k}^{n-1}\left(p_k\left(U^i\bar{f}_k-U^{d_k+i}\bar{f}_k\right)+p_{k+1}\left(U^i\bar{f}_{k+1}-U^{d_{k+1}+i}\bar{f}_{k+1}\right)\right). 
\] Finally, we set
\[
\mathsf{U}(f,n):=\sum_{k\in \hat{I}_n}V_k.
\]
In \cite{MR4374685} the local CLT is proved for the analogue of $\mathsf{U}(f,n)$ and then this is used for the deduction of local CLT for $S_n(f)$.
\begin{remark}\label{rem:change_of_U_term}
In \cite{MR4374685}, $I_n$ is defined as the set of all even integers $k$ such that $p_k<n<d_k$. After this the function $\mathsf{U}_n:=\sum_{k\in I_n}V_k$ is defined. 

Clearly for all $n\in \N$, $\hat{I}_n\subset I_n$. When $\log(n-1)\notin 2\N$,  $I_n$ and $\hat{I}_n$ coincide and $\mathsf{U}_n=\mathsf{U}(f,n)$. When $\log(n-1)\in 2\N$, then $I_n\setminus\hat{I}_n=\log(n-1)$ and
\[
\mathsf{U}_n-\mathsf{U}(f,n)=V_{\log(n-1)}.
\]
\end{remark}The following is a minor correction of the statement \cite[Proposition 11]{MR4374685}. We reproduce the correction of the proof here. 

\begin{prop}
\label{prop:third_bound_in_KV}
For every $n\in\mathbb{N}$, the random variables $\mathsf{U}(f,n)$ and 
$E_n:=\mathscr{W}_f(n)-\mathsf{U}(f,n)$
are independent and 
\[
\left\|E_n\right\|_2^2=O\left(\frac{n}{\sqrt{\log(n)}}\right). 
\] 
\end{prop}
\begin{proof}
There are three types of terms that appear in $\mathscr{W}_f(n)$ and not in $\mathsf{U}(f,n)$, and there are no terms that appear in $\mathsf{U}(f,n)$ and not in $\mathscr{W}_f(n)$. 

The first term comes from the case where for some even $k$, $d_k\leq n<d_{k+1}$. In this case $p_{k+1}<n<d_{k+1}$ and $B_{k+1}-U^{d_{k+1}}B_{k+1}$ appears in $\mathscr{W}_f(n)$ and not in $\mathsf{U}(f,n)$. When this term appears, then $\sqrt{\log(n)}<k+1$. As $B_{k+1}-U^{d_{k+1}}B_{k+1}$ is a sum of $n-p_{k+1}$ square integrable, zero mean functions, this implies that
\begin{align*}
1_{\left[d_k\leq n<d_{k+1}\right]}\left\|B_{k+1}-U^{d_{k+1}}B_{k+1}\right\|_2^2 &=\sum_{j=p_{k+1}-1}^{n-1}p_{k+1}^2\left\|U^j\left(\bar{f}_{k+1}-U^{d_{k+1}}\bar{f}_{k+1}\right)\right\|_2^2\\
&\leq 4np_{k+1}^2\left\|\bar{f}_{k+1}\right\|_2^2\\
&=4np_{k+1}^2\alpha_{k+1}^2=\frac{4n}{k+1}=O\left(\frac{n}{\sqrt{\log(n)}}\right).
\end{align*}
The second term appears when $\log(n-1)=k\in 2\N$. In this case $p_k<n=p_{k+1}<d_k$ so $B_k-U^{d_k}B_k$ appears in $\mathscr{W}_f(n)$ and not in $\mathsf{U}(f,n)$. Similarly to before,
\begin{align*}
1_{\left[p_k< n=p_{k+1}\right]}\left\|B_{k}-U^{d_{k}}B_{k}\right\|_2^2 &=\sum_{j=p_{k}-1}^{n-1}p_{k}^2\left\|U^j\left(\bar{f}_{k}-U^{d_{k}}\bar{f}_{k}\right)\right\|_2^2\\
&\leq 4np_{k}^2\left|\|\bar{f}_{k}\right\|_2^2\\
&=4np_{k}^2\alpha_{k}^2=\frac{4n}{k}=O\left(\frac{n}{\log(n)}\right).
\end{align*}
The third contribution to $E_n$ comes from the fact that for $k\in \hat{I}_n$, $p_k\left(U^{p_k-1}\bar{f}_k-U^{p_k-1+d_k}\bar{f}_k\right)$ appears in $B_k-U^{d_k}B_k$, hence in $\mathscr{W}_f(n)$, and not in $\mathsf{U}(f,n)$. 
We can conclude that\begin{equation}\label{eq:E_n_sum}
\begin{split}
E_n&=\sum_{k\in \hat{I}_n}p_k\left(U^{p_k-1}\bar{f}_k-U^{p_k-1+d_k}\bar{f}_k\right)+1_{\left[\exists k\in 2\N: d_k\leq n<d_{k+1}\right]}\left(B_{k+1}-U^{d_{k+1}}B_{k+1}\right) \\
&\ \ \ \ +1_{\left[\exists k\in 2\N:\ p_k<n = p_{k+1}\right]}\left(B_{k}-U^{d_{k}}B_{k}\right).
\end{split}
\end{equation}
Both $E_n$ and $\mathsf{U}(f,n)$ are functions of the independent\footnote{by properties \ref{item_b:bar_f} and \ref{item_c:bar_f} in the construction of $\bar{f}_k$.} sequence of functions
\[
\mathscr{L}:=\left\{\bar{f}_k\circ T^j: k\in \N, p_k<n<d_k,  0\leq j\leq 2d_k+p_k\right\}.
\]
In addition there exists two disjoint subsets $\mathscr{A},\mathscr{B}\subset\mathscr{L}$ such that $E_n$ is a function of $\mathscr{A}$ and $\mathsf{U}(f,n)$ is a function of $\mathscr{B}$. The independence of $E_n$ and $\mathsf{U}(f,n)$ follows from this as the functions in $\mathscr{L}$ are independent\footnote{In the proof of \cite[Proposition 11]{MR4374685} the problem is that $A(n)$ (see \cite[page 558]{MR4374685}) may appear in both $\mathrm{U}_n$ and $E_n$.}.

It is easy to see that the three terms in \eqref{eq:E_n_sum} above are independent and are square integrable and have zero mean, consequently 
\begin{align*}
\left\|E_n\right\|_2^2&=\left\|\sum_{k\in \hat{I}_n}p_k\left(U^{p_k-1}\bar{f}_k-U^{p_k-1+d_k}\bar{f}_k\right)\right\|_2^2+1_{\left[\exists !k\in 2\N: d_k\leq n<d_{k+1}\right]}\left\|B_{k+1}-U^{d_{k+1}}B_{k+1}\right\|_2^2\\
&\ \ \ \ +1_{\left[\exists k\in 2\N:\ p_k=n<p_{k+1}\right]}\left\|B_{k}-U^{d_{k}}B_{k}\right\|_2^2\\
&=\sum_{k\in \hat{I}_n}p_k^2\left\|\left(U^{p_k-1}\bar{f}_k-U^{p_k-1+d_k}\bar{f}_k\right)\right\|_2^2+O\left(\frac{n}{\sqrt{\log(n)}}\right)+O\left(\frac{n}{\log(n)}\right)\\
&\leq O\left(\frac{n}{\sqrt{\log(n)}}\right)+\sum_{k\in \hat{I}_n}4p_k^2\alpha_k^2\\
&\leq O\left(\frac{n}{\sqrt{\log(n)}}\right)+\sum_{k\in \hat{I}_n}\frac{4}{k}=O\left(\frac{n}{\sqrt{\log(n)}}\right).
\end{align*} \end{proof}

In \cite[Theorem 13]{MR4374685} one proves the local limit theorem for $\mathsf{U}_n$ using the fact that it satisfies the CLT (\cite[Corollary 12]{MR4374685}) together with two lemmas (\cite[Lemmas 14 and Lemma 15]{MR4374685}) regarding the Fourier transform of $\mathsf{U}_n$. In \Cref{Appendix B}, we prove analogous statements for $\mathsf{U}(f,n)$ (see \cref{thm:U_f_n_all_estimates}) and the following theorem.
\begin{thm}\label{thm:LCLT_for_U}
 Writing $\sigma^2=2(\ln 2)^2$ then,
 \[
 \sup_{x\in\Z}\left|\sqrt{n}\mu(\mathsf{U}(f,n)=x)-\frac{1}{\sqrt{2\pi\sigma^2}}e^{-\frac{x^2}{2n\sigma^2}}\right|=o(1). 
 \]
\end{thm}

\begin{proof} See \Cref{Appendix B}. 
\end{proof}

\subsection{Proof of Theorem \ref{thm:LCLT_1}}
Let $F:=\left(f^{(1)},\ldots,f^{(D)}\right)$ be the function as in \eqref{eq:def_function_F}. The proof starts by writing $S_n(F)$ as a sum of three terms depending on the scale of $k$ with respect to $n$. That is
\[
S_n(F)=Z_{SM}(F)(n)+\hat{Y}_F(n)+Z_{La}(F)(n),
\]
where 
\begin{align*}
Z_{SM}(F)(n)&:= \sum_{\{k: d_k\leq n\}}S_n\left(F_k\right)\\
\hat{Y}_F(n)&:= \sum_{\{k: p_k<n<d_k\}}S_n\left(F_k\right)\\
Z_{La}(F)(n)&:= \sum_{\{k: n\leq p_k\}}S_n\left(F_k\right). 
\end{align*}
Here $F_k$ is the function from \eqref{eq:def_function_F_k}. 
\begin{lemma}\label{lem:first_bound_in_KV_multi_d} For every $n\in\mathbb{N}$, the random variables $Z_{Sm}(F)(n),\hat{Y}_F(n),Z_{La}(F)(n)$ are independent and 
\[
\left\|Z_{Sm}(F)(n)+Z_{La}(F)(n)\right\|_2^2=O\left(\frac{n}{\sqrt{\log(n)}}\right). 
\]
\end{lemma}
\begin{proof}
For each $k\in \N$, $S_n(F_k)$ is a function of 
\[
\left\{\bar{f}_k^{(i)}\circ T^j:\ 1\leq i\leq D,\ 0\leq j\leq d_k+p_k+n-1\right\}.
\]
For all the $k$'s in the terms appearing in the sums of $Z_{La}(F)$ and $\hat{Y}_F(n)$ one has $n<d_k$. This implies that $\hat{Y}_F(n)$ is a function of 
\[
\left\{\bar{f}_k^{(i)}\circ T^j:\ 1\leq i\leq D,\ p_k<n<d_k,\ 0\leq j\leq 2d_k+p_k-1\right\},
\]
and $Z_{La}(n)$ is a function of 
\[
\left\{\bar{f}_k^{(i)}\circ T^j:\ 1\leq i\leq D,\ n\leq p_k,\ 0\leq j\leq 2d_k+p_k-1\right\}.
\]
Writing $k*$ for the first integer such that $d_{k}>n$, $Z_{SM}(F)$ is a function of 
\[
\left\{\bar{f}_k^{(i)}\circ T^j:\ 1\leq i\leq D,\ 0\leq j\leq 2d_{k*}+p_{k*}-1\right\}.
\]
The independence of $Z_{Sm}(F)(n),\hat{Y}_F(n)$ and $Z_{La}(F)(n)$ follows from property \ref{prop_sec_c:functions_defn} in \cref{prop:functions_defn}.

For $1\leq i\leq D$, the sequence of functions in the definition of $f^{(i)}$ satisfies conditions \ref{item_a:bar_f}, \ref{item_b:bar_f} and \ref{item_c:bar_f} as in \cref{subsec:llt dim1}. By \cref{lem:first_bound_in_KV}, for $1\leq i\leq D$,
\[
\left\|Z_{Sm}\left(f^{(i)}\right)(n)+Z_{La}\left(f^{(i)}\right)(n)\right\|_2^2=O\left(\frac{n}{\sqrt{\log(n)}}\right). 
\]
Finally
\begin{align*}
\left\|Z_{Sm}(F)(n)+Z_{La}(F)(n)\right\|_2^2&=\sum_{i=1}^D\left\|Z_{Sm}\left(f^{(i)}\right)(n)+Z_{La}\left(f^{(i)}\right)(n)\right\|_2^2\\
&=O\left(\frac{n}{\sqrt{\log(n)}}\right).
\end{align*}
\end{proof} For $k\in\N$, let $\bar{F}_k:=\left(\bar{f}_k^{(1)},\bar{f}_k^{(2)},\ldots,\bar{f}_k^{(D)}\right)$. For $n \in \N$ and $j\in\{1,\ldots,n\}$, set
\[
Y_j(F,n):=\sum_{\{k: p_k\leq j+1,\ p_k<n<d_k\}} p_k\left(U^j\bar{F}_k-U^{d_k+j}\bar{F}_k\right),
\]
and $\mathscr{W}_F(n):=\sum_{j=1}^{n-1} Y_j(F,n)$. Similar to the 1-dimensional case, writing 
\[
B_k(n):=p_k\sum_{j=p_k-1}^{n-1}U^j\bar{F}_k,
\]
we have 
\[
\mathscr{W}_F(n)=\sum_{\{k:\ p_k<n<d_k\}}\left(B_k(n)-U^{d_k}B_k(n)\right). 
\]
\begin{prop}
\label{prop:Similar_to_second_bound_in_KV}
For every $n\in\mathbb{N}$, the random variables $\mathscr{W}_F(n)$ and $S_n(F)-\mathscr{W}_F(n)$ are independent and 
\[
\left\|S_n(F)-\mathscr{W}_F(n)\right\|_2^2=O\left(\frac{n}{\sqrt{\log(n)}}\right). 
\]
\end{prop}
\begin{proof}
For all $n \in\mathbb{N}$,
\begin{align*}
\widehat{Y}_F(n)&=\mathscr{W}_F(n)+\sum_{\{k:\ p_k<n<d_k\}}\left(A_k(n)+C_k(n)-U^{d_k}\left(A_k(n)+C_k(n)\right)\right)\\
&=:\mathscr{W}_F(n)+\mathscr{Z}_F(n),
\end{align*}
where,
\begin{align*}
A_k(n)&:=\sum_{j=0}^{p_k-2}(j+1)U^j\bar{F}_k,\\
C_k(n)&:=\sum_{j=n}^{n+p_k-2}\left(n+p_k-1-j\right)U^j\bar{F}_k.
\end{align*}
Now $\mathscr{W}_F(n)$ is a function of 
\[
\mathcal{L}_n:=\left\{U^j\bar{f}_k^{(i)}:\ p_k<n<d_k,\ i\in\{1,\ldots,D\},\ p_k-1\leq j\leq n-1\right\},
\] and as $p_k<n$, $\mathscr{Z}_F(n)$ is a function of 
\[
\left\{U^j\bar{f}_k^{(i)}:\ p_k<n<d_k,\ i\in\{1,\ldots,D\},\ 0\leq j\leq 2d_k+p_k\right\}\setminus \mathcal{L}_n,
\] the independence of $Z_{SM}(F)(n)$, $Z_{La}(F)(n)$, $\mathscr{Z}_F(n)$ and $\mathscr{W}_F(n)$ follows from properties \ref{prop_sec_b:functions_defn} and \ref{prop_sec_c:functions_defn} in \cref{prop:functions_defn}. As 
\begin{equation}\label{eq:S_n-W_n}
S_n(F)-\mathscr{W}_F(n)=Z_{SM}(F)(n)+Z_{La}(F)(n)+\mathscr{Z}_F(n)
\end{equation}
we have established the independence of $S_n(F)-\mathscr{W}_F(n)$ and $\mathscr{W}_F(n)$. For all $k\in\N$,
\begin{align*}
\left\|\bar{F}_k\right\|_2^2=\sum_{i=1}^D \left\|\bar{f}_k^{(i)}\right\|_2^2 
=D\mu\left(\bar{f}_k^{(1)}\neq 0\right)=D\alpha_k^2. 
\end{align*}
It follows from this and $\int F_k d\mu=0$ that 
\[
\left\|A_k(n)\right\|_2^2=\sum_{j=0}^{p_k-2}(j+1)^2\left\|F_k\right\|_2^2\leq Dp_k^3\alpha_k^2.
\]
Similarly, 
\[
\left\|C_k(n)\right\|_2^2=\sum_{j=n}^{n+p_k-2}\left(n+p_k-1-j\right)^2\left\|F_k\right\|_2^2\leq Dp_k^3\alpha_k^2.
\]
The collection of functions $A_k(n)$, $C_k(n)$, $U^{d_k}A_k(n)$, $U^{d_k}C_k(n)$ with $k$ in the range $p_k<n<d_k$ are independent and with integral $0$, consequently
\begin{align*}
\left\|\mathscr{Z}_F(n)\right\|_2^2&=\sum_{k:\ p_k<n<d_k}\left(\left\|A_k(n)\right\|_2^2+\left\|C_k(n)\right\|_2^2+\left\|U^{d_k}A_k(n)\right\|_2^2+\left\|U^{d_k}C_k(n)\right\|_2^2\right)\\
&\leq 4D\sum_{k:\ p_k<n<d_k}p_k^3\alpha_k^2\\
&\leq 4D\sum_{k:\ p_k<n<d_k}\frac{2^k}{k}=O\left(\frac{n}{\log(n)}\right).
\end{align*}
Taking in view \eqref{eq:S_n-W_n} and \cref{lem:first_bound_in_KV_multi_d}, we see that
\[
\left\|S_n(F)-\mathscr{W}_F(n)\right\|_2^2=O\left(\frac{n}{\sqrt{\log(n)}}\right).
\]\end{proof}
Recall that $\hat{I}_n:=\{k\in 2\N:\ p_{k+1}<n<d_k\}$. For $k\in \hat{I}_n$, let 
\[
V_{k}(F,n):=\sum_{i=2^k}^{n-1}\left(p_k\left(U^i\bar{F}_k-U^{d_k+i}\bar{F}_k\right)+p_{k+1}\left(U^i\bar{F}_{k+1}-U^{d_{k+1}+i}\bar{F}_{k+1}\right)\right). 
\]
Define,
\begin{equation}\label{eq: formula for U(F)}
\mathsf{U}(F,n):=\sum_{k\in \hat{I}_n}V_k(F,n)=\left(\mathsf{U}\left(f^{(1)},n\right),\mathsf{U}\left(f^{(2)},n\right),\ldots,\mathsf{U}\left(f^{(D)},n\right)\right).    
\end{equation}
The following is a multi-dimensional analogue of \cref{prop:third_bound_in_KV}.
\begin{prop}\label{prop:third_bound_in_KV_multi_D}
For every $n\in\mathbb{N}$, the random variables $\mathsf{U}(F,n)$, $S_n(F)-\mathscr{W}_F(n)$ and 
$E_n:=\mathscr{W}_F(n)-\mathsf{U}(F,n)$
are independent and 
\[
\left\|E_n\right\|_2^2=O\left(\frac{n}{\sqrt{\log(n)}}\right). 
\]
\end{prop} 
\begin{proof}
The independence of $\mathsf{U}(F,n)$ and $E_n$ from $S_n(F)-\mathscr{W}_F(n)$ follows from the definition of $\mathsf{U}(f,n)$ and (the argument of) \cref{prop:Similar_to_second_bound_in_KV}. Arguing as in the proof of \cref{prop:third_bound_in_KV} one deduces the independence of $\mathsf{U}(F,n)$ and $E_n$ and the bound of $\left\|E_n\right\|_2^2$.
\end{proof}
We now turn to show the following local CLT for $\mathsf{U}(F,n)$. 
\begin{thm}\label{thm:LCLT_for_U_F_n}
\[
\sup_{x\in\Z^D} \left|n^{D/2}m\left(\mathsf{U}(F,n)=x\right)-\frac{1}{(2\pi \sigma^2)^{D/2}}e^{-\frac{\|x\|^2}{2n\sigma^2}}\right|\xrightarrow[n\to\infty]{}0.
\]
where $\sigma^2=2\left(\ln 2\right)^2$. 
\end{thm}
One step of the proof will make use of the following simple claim.
\begin{claim}\label{clm:triangle}
For every $z_1,\ldots,z_m,y_1,\ldots, y_m$ real numbers such that $\max_{1\leq m}|z_i|,\max_{1\leq m}|y_i|\leq C$,
\[
\left|\prod_{i=1}^m z_i-\prod_{i=1}^m y_i\right|\leq \sum_{i=1}^m C^{m-1}\left|z_i-y_i\right|.
\]
\end{claim}
\begin{proof}[Proof of \cref{thm:LCLT_for_U_F_n}]
Fix $1\leq i\leq D$. By construction, the functions $\bar{f}_k^{(i)}:X\to \R$, satisfy conditions \ref{item_a:bar_f},\ref{item_b:bar_f} and \ref{item_c:bar_f} in Subsection \ref{subsec:llt dim1}. By \cref{thm:LCLT_for_U},
\[
\sup_{x\in\Z}\left|\sqrt{n}\mu\left(\mathsf{U}\left(f^{(i)},n\right)=x\right)-\frac{e^{-\frac{x^2}{2n\sigma^2}}}{\sqrt{2\pi\sigma^2}}\right|=o(1).
\]
We conclude that there exists $0<r_n\to 0$ so that
\begin{equation}\label{eq:LCLT_for_all_functions}
\max_{1\leq i\leq D}\sup_{x\in\Z}\left|\sqrt{n}\mu\left(\mathsf{U}\left(f^{(i)},n\right)=x\right)-\frac{e^{-\frac{x^2}{2n\sigma^2}}}{\sqrt{2\pi\sigma^2}}\right|\leq r_n
\end{equation}
and there exists $C>0$ such that for all large $n$, 
\begin{equation}\label{eq:bound_from_LCLT}
\text{for all}\  x\in \Z,\ \max_{1\leq i\leq D}\left(\sqrt{n}\mu\left(\mathsf{U}\left(f^{(i)},n\right)=x\right)\right)\leq C.
\end{equation}
Now $\mathsf{U}\left(f^{(i)},n\right)$ is a function of 
\[
\left\{U^j\bar{f}_k^{(i)}:\ p_k<n<d_k,\ 0\leq j\leq 2d_k+p_k\right\}. 
\]
We may further assume that $C>\frac{1}{\sqrt{2\pi\sigma^2}}=\max_{x\in\Z}\frac{e^{-\frac{x^2}{2n}}}{\sqrt{2\pi\sigma^2}}$. By properties \ref{prop_sec_b:functions_defn} and \ref{prop_sec_c:functions_defn} in \cref{prop:functions_defn}, $\mathsf{U}\left(f^{(1)},n\right),\dots,\mathsf{U}\left(f^{(D)},n\right)$ are independent. By this and \cref{clm:triangle},  for all $x\in \Z^D$,
\begin{align*}
\left|n^{D/2}m\left(\mathsf{U}(F,n)=x\right)-\frac{1}{(2\pi \sigma^2)^{D/2}}e^{-\frac{\|x\|^2}{2n\sigma^2}}\right|&=\left|\prod_{i=1}^D\sqrt{n}m\left(\mathsf{U}\left(f^{(i)},n\right)=x_i\right)-\frac{e^{-\sum_{i=1}^D\frac{x_i^2}{2n\sigma^2}}}{(2\pi \sigma^2)^{D/2}}\right|\\
&\leq \sum_{i=1}^DC^{D-1}\left|\sqrt{n}m\left(\mathsf{U}\left(f^{(i)},n\right)=x_i\right)-\frac{e^{-\frac{x_i^2}{2n\sigma^2}}}{\sqrt{2\pi\sigma^2}}\right|
\end{align*}
where the last inequality is a routine application of the triangle inequality and \eqref{eq:bound_from_LCLT}. Taking into account \eqref{eq:LCLT_for_all_functions} we have shown that for all $x\in \Z^D$, 
\[
\left|n^{D/2}m\left(\mathsf{U}(F,n)=x\right)-\frac{1}{(2\pi \sigma^2)^{D/2}}e^{-\frac{\|x\|^2}{2n\sigma^2}}\right|\leq DC^{D-1}r_n=o(1). 
\]
This concludes the proof. 
\end{proof}
\begin{proof}[Proof of \cref{thm:LCLT_1}]
Write 
\begin{align*}
S_n(F)&=\mathsf{U}(F,n)+\left(S_n(F)-\mathsf{U}(F,n)\right)\\
&=\mathsf{U}(F,n)+\left(S_n(F)-\mathscr{W}_F(n)\right)+\left(\mathscr{W}_F(n)-\mathsf{U}(F,n)\right)
\end{align*}

By \cref{prop:third_bound_in_KV_multi_D},  $\mathsf{U}(F,n)$ and $S_n(F)-\mathsf{U}(F,n)$ are independent. In addition by \cref{prop:Similar_to_second_bound_in_KV} and \cref{prop:third_bound_in_KV_multi_D},
\[
\left\|S_n(F)-\mathsf{U}(F,n)\right\|_2^2=O\left(\frac{n}{\sqrt{\log(n)}}\right). 
\]
The claim now follows from \cref{thm:LCLT_for_U_F_n} and \cref{prop:bla}.
\end{proof}

\section{Divergence of non-conventional ergodic averages}
In this section, we prove \cref{thm:Main}. Our proof uses an idea of Huang, Shao, and Ye from \cite{Huang_Shao_Ye_2024}, which is to consider $T$ a skew product extension of an irrational rotation by a function satisfying the local central limit theorem. After this, they define $S$ as a carefully chosen transformation that is isomorphic to $T$. 

There are two notable differences in our construction. First, our skew product extension is an extension by a $\Z^2$ full-shift, and the function satisfies the $2$-dimensional local central limit theorem. The second is that, if we would then proceed similarly to \cite{Huang_Shao_Ye_2024}, we would obtain the result for polynomials of degree $3$ or higher. In order to also include polynomials of degree $2$, we define $S$ by using a detailed study of the range process of the cocycle at polynomial times (see Subsection \ref{sec:construction_S}). 

\subsection{Construction of the ergodic system $(X,\mathcal{X},\mu, T)$}\label{sec:construction_T}

We denote by $\T = \R /\Z$ the  unit circle. For $\alpha \in \R/\mathbb{Q}$, let $R_{\alpha}: \T \rightarrow \T$ be the irrational rotation defined by $R_{\alpha} (y) = (y+\alpha)\,\mathrm{mod}\,1$. Let $\lambda$ be the Lebesgue measure on $\T$. Thus $(\T,\B(\T), \lambda, R_{\alpha})$ is an ergodic measure preserving system, where $\B(\T)$ is the Borel sigma algebra generated by open sets in $\T$. 

By \Cref{thm:LCLT} for $d=2$, there exists a Borel function, $f: \T \rightarrow \Z^2$, given by $f (y) = (f^{(1)}(y), f^{(2)}(y))$ for $y \in \T$, such that the corresponding $2$-dimensional ergodic sums process $S_n(f) : \T \rightarrow \Z^2$, given by
\[
S_n(f) (y):=\sum_{k=0}^{n-1}f\circ R_{\alpha}^k (y),
\]
satisfies the lattice local central limit theorem. In this section, we fix such a Borel function $f$.

Let $\Sigma = \{0,1\}^{\Z^2}$ be the space of $2$-dimensional arrays of $\{0,1\}$. For $v \in \Z^2$, we denote by $\sigma_v : \Sigma \rightarrow \Sigma$ the $2$-dimensional shift, given by for every $\omega\in \Sigma$
\[
(\sigma_v \, \omega) (u) = \omega (u+v).
\] 
We endow $\Sigma$ with the stationary (infinite) product measure $\nu=\left(\frac{1}{2}\delta_0+\frac{1}{2}\delta_1\right)^{\Z^2}$ with marginals $\left(\frac{1}{2},\frac{1}{2}\right)$, and consider the $\Z^2$ Bernoulli shift $(\Sigma, \B(\Sigma), \nu, \sigma)$. Here $\B(\Sigma)$ is the Borel sigma algebra generated by cylinder sets in $\Sigma$. We define $(X,\mathcal{X},\mu)$ to be the Cartesian product space $\T \times \Sigma$, endowed with product measure, in other words
\[(X,\mathcal{X},\mu) : = (\T \times \Sigma, \B(\T) \otimes \B(\Sigma), \lambda \times \nu).\]

Let  $T : \T \times \Sigma \rightarrow \T \times \Sigma$ be the skew product of $R_\alpha$ and $f$ defined by 
\begin{equation*}
    T (y, \omega) = (R_{\alpha} (y), \sigma_{f(y)} (\omega)).
\end{equation*}
The skew product $T$ is a measure preserving transformation of $(X,\mathcal{X},\mu)$ and for all $n\in\N$,
\begin{equation}\label{eq:T^n}
    T^n (y, \omega) = (R^n_{\alpha} (y), \sigma_{S_n(f)(y)} (\omega)).
\end{equation}
For $v = (v^1,v^2) \in \N^2$, we define the rectangle centered at the origin with side lengths $(2v^1+1)$ and $(2v^2+1)$ by,
\begin{equation}\label{eq:rectangle}
    U_{v} : =  \{-v^1,\ldots,0,\ldots,v^1\} \times \{-v^2,\ldots,0,\ldots,v^2\}. 
\end{equation}
For $a \in \{0,1\}^{U_v}$, we denote by $[a]_{v}$ the cylinder set defined by $a$ and $U_v$
\begin{equation}\label{eq:cylinder_set}
[a]_{v} : = \{\omega \in \Sigma: \omega(i,j) = a(i,j),\,\, \forall (i,j) \in U_v\}.    
\end{equation} 

\begin{prop}\label{prop:T_ergodic} $(X,\mathcal{X},\mu, T)$ is an ergodic measure preserving system.
\end{prop}

\begin{proof} Fix $B_1, B_2 \in \B(\T)$ such that $\lambda(B_1)>0$ and $\lambda(B_2)>0$. For $v=(v^1,v^2) \in \N^2$ and $u = (u^1,u^2) \in \N^2$, fix cylinder sets $[a_1]_{v}$ and $[a_2]_{u}$. Denote by $m(u,v) \in \N^2$ the vector \[
m(u,v) : = 2\cdot (\mathrm{max}\{u^1+1,v^1+1\}, \mathrm{max}\{u^2+1,v^2+1\}).
\] For $n \in \N$, we define
\[
M_n = \{y \in \T: \|S_n(f) (y)\|_{\infty} \leq \| m(u,v)\|_{\infty}\}. 
\] 
Thus,
\begin{align*}
\mu(B_1 \times [a_1]_v \cap T^{-n}(B_2 \times [a_2]_u)) 
= \int_{\Sigma}\int_{\T} 1_{B_1 \cap R^{-n}_{\alpha}B_2} (y) \cdot 1_{\sigma_{-S_n(f) (y) [a_2]_u \cap [a_1]_v} }(\omega) d\lambda(y)\,d\nu(\omega).
\end{align*} 
For all $p \in \Z^2$ such that $\|p\|_{\infty} > \|m(u,v)\|_{\infty}$,  
\[
\nu(\sigma_{-p}[a_2]_u \cap [a_1]_v) = \nu([a_2]_u)\cdot \nu([a_1]_v).
\]
Thus, 
\begin{align*}
\mu(B_1 \times [a_1]_v \cap T^{-n}(B_2 \times [a_2]_u)) &\geq \int_{\Sigma}\int_{\T\setminus M_n} 1_{B_1 \cap R^{-n}_{\alpha}B_2} (y) \cdot 1_{\sigma_{-S_n(f) (y) [a_2]_u \cap [a_1]_v} }(\omega) d\lambda(y)\,d\nu(\omega)\\
&= \int_{\T\setminus M_n} 1_{B_1 \cap R^{-n}_{\alpha}B_2} (y) \cdot \nu([a_2]_u) \nu([a_1]_v) d\lambda(y)\\
&\geq (\lambda (B_1 \cap R^{-n}_{\alpha}B_2) - \lambda(M_n)) \cdot \nu([a_2]_u) \nu([a_1]_v).
\end{align*}
Thus we get,
\begin{equation}\label{eq:mu_arbitary}
\mu(B_1 \times [a_1]_v \cap T^{-n}(B_2 \times [a_2]_u)) \geq (\lambda (B_1 \cap R^{-n}_{\alpha}B_2) - \lambda(M_n)) \cdot \nu([a_2]_u) \nu([a_1]_v).
\end{equation} By \Cref{thm:LCLT}, for $d=2$ and $(i,j) \in \Z^2$, there exits $N(i,j) \in \N$ such that for all $n> N(i,j)$, we have 
\[
\left|n \cdot \lambda\left(S_n(f)=(i,j)\right)-\frac{e^{-\frac{\|(i,j)\|^2}{2n\sigma^2}}}{(2\pi \sigma^2)}\right| < \frac{1}{(2\pi \sigma^2)}.
\] Hence for all $n> N(i,j)$,
\[
\lambda\left(S_n(f)= (i,j)\right) < \dfrac{e^{-\frac{\|(i,j)\|^2}{2n\sigma^2}}}{n(2\pi\sigma^2)} + \dfrac{1}{n(2\pi \sigma^2)} \leq \dfrac{2}{n(2\pi \sigma^2)}. 
\] Let $K = \underset{\{ (i,j) : \|(i,j)\|_{\infty} \leq \|m(u,v)\|_{\infty}\}}{\mathrm{max}} N(i,j)$. Then for any $n> K$, we have
\[
\lambda(M_n) =  \underset{\{ (i,j) \in \Z^2: \|(i,j)\|_{\infty} \leq \|m(u,v)\|_{\infty}\}}{\sum} \lambda\left(S_n(f) =(i,j)\right)  \leq \dfrac{(2 \cdot (\|m(u,v)\|_{\infty}+1))^2}{n(2\pi \sigma^2)}.
\] 
Thus $\underset{n\rightarrow \infty}{\mathrm{lim}} \lambda(M_n) = 0$ and by ergodicity of $R_\alpha$,
\[
\underset{N\rightarrow \infty}{\mathrm{lim}} \dfrac{1}{N} \sum^{N-1}_{n=0} (\lambda (B_1 \cap R^{-n}_{\alpha}B_2) - \lambda(M_n)) =\underset{N\rightarrow \infty}{\mathrm{lim}} \dfrac{1}{N} \sum^{N-1}_{n=0} \lambda (B_1 \cap R^{-n}_{\alpha}B_2)=\lambda\left(B_1\right)\lambda\left(B_2\right).  
\]
We deduce from this and \eqref{eq:mu_arbitary} that, 
\begin{align*}
\underset{N\rightarrow \infty}{\mathrm{lim}} \dfrac{1}{N} \sum^{N-1}_{n=0} \mu(B_1 \times [a_1]_v \cap T^{-n}(B_2 \times [a_2]_u))
&\geq \lambda(B_1)\lambda(B_2)\cdot \nu([a_2]_u) \nu([a_1]_v)\\
&\geq \mu (B_1 \times [a_1]_v) \cdot \mu (B_2 \times [a_2]_v).
\end{align*}
Since $B_1,B_2 \in \B(\T)$ were arbitrary positive $\lambda$-measures sets, and $[a_1]_v, [a_2]_u \in \Sigma$ were arbitrary cylinder sets and $T$ is measure preserving, it follows that for any positive $\mu$-measure $A,B \in \mathcal{X}$, we have
\[
\underset{N\rightarrow \infty}{\mathrm{lim}} \dfrac{1}{N} \sum^{N-1}_{n=0} \mu (A \cap T^{-n} B) \geq \mu(A)\cdot \mu(B). 
\] In other words for any positive $\mu$-measure $A,B \in \mathcal{X}$, there exits $n \in \N$ such that $\mu (A \cap T^{-n} B) \geq 0$. Hence $(X,\mathcal{X},\mu, T)$ is ergodic. \end{proof}
For $y \in \T$ and $N \in \N$, consider the set \[
A_N(y) = \{ S_n(f)(y)  :0\leq n \leq N-1 \} \subset \Z^2.
\] Note that for $y\in \T$, the cardinality of $A_N(y)$ can be at most $N$. 

\begin{lemma}\label{lem:A_n_y_limit} For a.e. $y\in \T$, 
\[
\underset{N\rightarrow \infty}{\mathrm{lim}} \dfrac{|A_N(y)|}{N} = 0.
\]
\end{lemma}
The proof of the lemma uses the notion of recurrent cocycles. Given a probability preserving system $(X,\B,\mathbb{P},R)$ and a function $h:X\to \R^d$, its corresponding cocycle is \textbf{recurrent} if for every $B\in\B$, and $\epsilon>0$, there exists $n\in\N$, such that
\[
\mathbb{P}\left(B\cap R^{-n}B\cap \{|S_n(h)|<\epsilon\}\right)>0.
\]
\begin{proof}
Since $S_n(f)$ satisfies the $2$-dimensional local central limit theorem, $S_n(f)/\sqrt{n}$ converges weakly to a $2$-dimensional normal distribution. By \cite{MR1663750,MR1721618}\footnote{Note that $||f||\in L^2(\lambda)$ so our cocycle also satisfies the extra condition in \cite{MR1721618}} the cocycle $S_n(f)$ is recurrent. Now as $S_n(f)$ is $\Z^2$ valued, for every $B\in \B(\T)$, there exists $n\in \N$, such that 
\[
\lambda\left(B\cap R_\alpha^{-n}B\cap \{S_n(f)=(0,0)\}\right)>0.
\]
This implies that 
\[
\lambda\left(x\in\mathbb{T}:\ \forall n\in\N, S_n(f)\neq (0,0)\right)=0. 
\]
The claim now follows from \cite[Proposition 2.1]{MR4290512}.
\end{proof} Now we show that the measure theoretic entropy of $(X,\mathcal{X},\mu, T)$ is zero. 

\begin{prop}\label{prop:entropy_T=0} $h_{\mu}(X,T) = 0$. 
    
\end{prop}

\begin{proof} For any finite partition $\xi$ of $\Sigma = \{0,1\}^{\Z^2}$, \[
h_{\mu} (T\mid R_{\alpha}, \xi) = \underset{N\rightarrow \infty}{\mathrm{lim}} \dfrac{1}{N} \int_{\T} H_{\nu} (\bigvee^{N-1}_{n=0} \sigma_{-S_n(f)(y)} \xi ) d\lambda(y),
\]
where we set $S_0(f) := (0,0)$. Note that for $y \in \T$, the cardinality of $\bigvee^{N-1}_{n=0} \sigma_{-S_n(f)(y)} \xi$ is bounded above by $|\xi|^{|A_n(y)|}$. Thus,\[
h_{\mu} (T\mid R_{\alpha}, \xi) \leq \underset{N\rightarrow \infty}{\mathrm{lim}} \dfrac{1}{N} \int_{\T} \mathrm{log} |\xi|^{|A_n(y)|}  d\lambda(y) = \underset{N\rightarrow \infty}{\mathrm{lim}}  \int_{\T} \dfrac{|A_N(y)|}{N} \mathrm{log} |\xi| d\lambda(y),
\] 
\[
= \mathrm{log} |\xi| \underset{N\rightarrow \infty}{\mathrm{lim}}  \int_{\T} \dfrac{|A_N(y)|}{N}  d\lambda(y).
\] We now apply the Dominated Convergence Theorem and use \cref{lem:A_n_y_limit} to obtain,
\[
h_{\mu} (T\mid R_{\alpha}, \xi) \leq 
\mathrm{log} |\xi| \underset{N\rightarrow \infty}{\mathrm{lim}}  \int_{\T} \dfrac{|A_N(y)|}{N}  d\lambda(y) = \mathrm{log} |\xi| \int_{\T} \underset{N\rightarrow \infty}{\mathrm{lim}} \dfrac{|A_N(y)|}{N} d\lambda(y) = 0. 
\] Since \[h_{\mu} (T\mid R_{\alpha}) = \underset{\xi}{\mathrm{sup}}\, h_{\mu} (T\mid R_{\alpha}, \xi), \] where the supremum is taken over all finite measurable partitions of $\Sigma$, we get $h_{\mu} (T\mid R_{\alpha}) = 0$. By the Abramov-Rokhlin formula \cite{Abramov_Rohlin_1972},
\[
h_{\mu}(X,T) = h_{\lambda}(\T,R_{\alpha}) + h_{\mu} (T\mid R_{\alpha}) = 0.
\] \end{proof}
\subsection{Construction of the ergodic system $(X,\mathcal{X},\mu, S)$}\label{sec:construction_S} In this subsection, we discuss the construction of $(X,\mathcal{X},\mu, S)$ and show that it is isomorphic to the system $(X,\mathcal{X},\mu, T)$ constructed in Subsection \ref{sec:construction_T}. Let $f:\mathbb{T}\to\Z^2$ be the function as in Subsection \ref{sec:construction_T}. In other words the $2$-dimensional ergodic sums process $S_n(f) : \T \rightarrow \Z^2$, satisfies the lattice local central limit theorem (\Cref{thm:LCLT}). Let $p:\Z\to\Z$, be a polynomial, for $y \in \T$ and $N\in\N$, we set \begin{equation}\label{eq:R_N_p_y}
\mathbf{R}_N^{(p)}(y):=\left\{S_{p(k)}(f)(y):\ 1\leq k\leq N\right\} \subset \Z^2,
\end{equation} and 
\begin{equation}\label{eq:R_p_y}
\mathbf{R}^{(p)}(y):=\left\{S_{p(k)}(f)(y):\ k\in\N\right\} \subset \Z^2.  
\end{equation} In rest of the section, we will work with polynomials with positive leading coefficient. Note that if the statement of \cref{thm:Main} holds for polynomials with positive leading coefficient than it also holds for polynomials with negative leading coefficient by replacing $T$ or $S$ (or both) with $T^{-1}$ and $S^{-1}$ respectively.

\begin{prop}\label{prop:Range_of_cocycle_along_p} Let $p:\Z\to\Z$ be a polynomial with a positive leading coefficient and $\deg(p)\geq 2$, then for Lebesgue almost every $y \in \T$,
\[
\lim_{n\to\infty} \dfrac{\left|\mathbf{R}_{n}^{(p)}(y)\right|}{n} = 1.
\]
\end{prop} The proposition essentially follows from the following lemma. 
\begin{lemma}\label{lem:expectation_and_variance_bound_on_range}
Let $p:\Z\to\Z$ be a polynomial with a positive leading coefficient and $\deg(p)\geq 2$, then we have

\begin{enumerate}[label=(\alph*)]
\item\label{item_a:R_n_estimate} $\underset{n\to\infty}{\lim}\frac{\mathbb{E}\left(\left|\bR_{n}^{(p)}\right|\right)}{n}=1$.
\item\label{item_b:R_n_estimate}  There exists $K>0$ and $M\in \N$, such that for all $n > M$, we have
    \[
   n-K\sqrt{n}\leq \mathbb{E}\left(\left|\bR_{n}^{(p)}\right|\right)\leq n.
    \]
\item\label{item_c:R_n_estimate}  There exists $C>0$  such that for all $n\in\N$,
    \[
\mathrm{Var}\left(\left|\bR_n^{(p)}\right|\right)\leq C \, n^{\frac{3}{2}}.
    \]
\end{enumerate}

\end{lemma}
\begin{proof}[Proof of \Cref{prop:Range_of_cocycle_along_p}]
Let $k \in \N$ and set $n_k=k^4$. We define,
\[
B_k:=\left\{y\in\T:\ \left|\frac{\left|\bR_{n_k}^{(p)}(y)\right|}{n_k}-1\right|>k^{-1/4}\right\}.
\]
By \Cref{lem:expectation_and_variance_bound_on_range}.\ref{item_b:R_n_estimate}, for $n_k > M$, there exists $K > 0$ so that
\[
\left|\frac{\mathbb{E}\left(\left|\bR_{n_k}^{(p)}\right|\right)}{n_k}-1\right|\leq \frac{K}{k^2}.
\] Hence for all $k \in \N$ such that $k > 16 K^4$ we have,
\begin{align*}
B_k& \subset \left\{y\in \T: \left|\frac{\left|\bR_{n_k}^{(p)}(y)\right|-\mathbb{E}\left(\left|\bR_{n_k}^{(p)}\right|\right)}{n_k}\right|>\frac{1}{2}k^{-1/4} \right\}.
\end{align*}
Finally using Markov's inequality, \Cref{lem:expectation_and_variance_bound_on_range}.\ref{item_c:R_n_estimate} and $n_k = k^4$, we get
\begin{align*}
\lambda\left(B_k\right)&\leq \lambda\left(\left|\frac{\left|\bR_{n_k}^{(p)}(y)\right|-\mathbb{E}\left(\left|\bR_{n_k}^{(p)}\right|\right)}{n_k}\right|>\frac{1}{2}k^{-1/4} \right)\\
&\leq \frac{4\mathrm{Var}\left(\left|\bR_{n_k}^{(p)}\right|\right)}{(n_k)^2 k^{-1/2}}\leq 4Ck^{-\frac{3}{2}}.
\end{align*} Here $C>0$ is as in \Cref{lem:expectation_and_variance_bound_on_range}\ref{item_c:R_n_estimate}. Thus it follows $\sum_{k=1}^{\infty} \lambda\left(B_k\right) < \infty$. We conclude from the Borel–Cantelli lemma that 
\[
\lambda \,\left(\underset{n\rightarrow\infty}{\limsup}\,\, B_k\right) = 0.
\] 
Thus for almost every $y \in \T$, we have
\begin{equation}\label{eq:lim_n_k}
 \lim_{k\to\infty} \frac{\left|\bR_{n_k}^{(p)}(y)\right|}{n_k}=1.  
\end{equation} Observe that for $y \in \T$, the sequence $n\mapsto \left|R_{n}^{(p)}(y)\right|$ is monotone increasing. Thus for any $n \in \N$ with $n_k\leq n<n_{k+1}$, we get 
\[
\frac{n_k}{n_{k+1}}\cdot\frac{\left|\bR_{n_k}^{(p)}(y)\right|}{n_{k}}\leq\frac{\left|\bR_{n_k}^{(p)}(y)\right|}{n_{k+1}}\leq \frac{\left|\bR_{n}^{(p)}(y)\right|}{n}\leq \frac{\left|\bR_{n_{k+1}}^{(p)}(y)\right|}{n_{k}}\leq \frac{n_{k+1}}{n_k}\cdot\frac{\left|\bR_{n_{k+1}}^{(p)}(y)\right|}{n_{k+1}}
\]
By \eqref{eq:lim_n_k}, for almost every $y \in \T$, we have
\[
\lim_{k\to\infty}\left(\frac{n_k}{n_{k+1}}\cdot\frac{\left|\bR_{n_k}^{(p)}(y)\right|}{n_{k}}\right)=\lim_{k\to\infty}\left(\frac{n_{k+1}}{n_k}\cdot\frac{\left|\bR_{n_{k+1}}^{(p)}(y)\right|}{n_{k+1}}\right)
=1.
\] This completes the proof of the proposition. \end{proof}
To prove \cref{lem:expectation_and_variance_bound_on_range}, we will need the following simple claim.
\begin{claim}\label{claim:simple_bound}
Let $p:\Z\to\Z$ be a polynomial with a positive leading coefficient and $\deg(p)\geq 2$. Then there exists $N\in\mathbb{N}$ and $\gamma>0$ so that for every $n>N$ and $1<k<n$,
\[
p(n)-p(k)\geq \gamma\big(n+(n-k)^2\big).
\]
\end{claim}
\begin{proof}
Let $p(n) = \sum_{i=0}^{t} c_i n^{i}$, where $t\geq 2$ and $c_t > 0$. Set $b_n = p(n+1) - p(n)$, then
\[
\lim_{n\to\infty} \dfrac{b_n}{(n+1)^t - n^t} =\lim_{n\to\infty} \dfrac{p(n+1) - p(n)}{(n+1)^t - n^t} = c_t >0.
\]
Thus there exists $M\in \N$ such that for $n \geq M$,
\[
b_n \geq \dfrac{c_t}{2} ((n+1)^t - n^t) > 0.
\] For  $n>k\geq M$, 
\begin{align*}
 p(n) - p(k) =  \sum_{j=k}^{n-1} b_{j} &\geq \sum_{j=k}^{n-1} \dfrac{c_t}{2} ((j+1)^t - j^t) \\
 &=\frac{c_t}{2}\left(n^t-k^t\right)\\
 &=\frac{c_t}{2}(n-k)\left(\sum_{l=0}^{t-1}k^ln^{t-1-l}\right).
\end{align*}
In addition,
\begin{align*}
\sum_{l=0}^{t-1}k^ln^{t-1-l}&\geq n^{t-1}+k^{t-1}\\
&\geq n+k=2k+(n-k).
\end{align*}
We deduce that for all $n>k\geq M$,
\begin{align*}
p(n) - p(k)&\geq \frac{c_t}{2}\left((n-k)^2+k(n-k)\right)\\
&\geq \frac{c_t}{2}\left((n-k)^2+\frac{n}{2}\right).
\end{align*}
Note that the last inequality holds as $k,(n-k)\geq 1$ and at least one of them is no smaller than $\frac{n}{2}$. In addition,  
\[
\lim_{n\to\infty}\frac{\min_{1\leq j< M}|p(n)-p(j)|}{n^t}=c_t>0
\]
and $n^2+n=O\left(n^t\right)$. The claim follows by a standard argument.
\end{proof}
\begin{proof}[Proof of \Cref{lem:expectation_and_variance_bound_on_range}]
For $2\leq k\in \N$, set 
\[
A_k:=\left\{x\in\T:\ \forall l\in\{1,\ldots,k-1\},\ S_{p(k)}(f)(x)\neq S_{p(l)}(f)(x)\right\}.
\]
Note that for $y \in \T$, we have
\[
\left|\bR_n^{(p)}(y)\right|:=1+\sum_{k=2}^n1_{A_k}(y).
\]
For every $l<k$, 
\[
S_{p(k)}(f)-S_{p(l)}(f)=S_{p(k)-p(l)}(f)\circ R_\alpha^{p(l)}
\] Now using  the local central limit theorem (\cref{thm:LCLT}) for $d=2$, this implies the existence of a constant $\beta>0$, such that for every $1\leq l<k$, we have
\begin{align*}
\lambda\left(S_{p(k)}(f)-S_{p(l)}(f)=(0,0)\right)&=\lambda\left(S_{p(k)-p(l)}(f)\circ R_\alpha^{p(l)}=(0,0)\right)\\
&=\lambda\left(S_{p(k)-p(l)}(f)=(0,0)\right)\leq \frac{\beta}{p(k)-p(l)}.
\end{align*}
As $\deg(p)\geq 2$, it follows from \Cref{claim:simple_bound}, that there exists $M\in\N$ and $c>0$ such that for all $n>M$ and $1\leq k<n$, we have
\begin{equation}\label{eq:important_bound_for_range}
\lambda\left(S_{p(n)}(f)=S_{p(k)}(f)\right)\leq \frac{c}{n+(n-k)^2}.
\end{equation}
For every $N> M$, 
\begin{align*}
\lambda\left(\T\setminus A_N\right)&\leq \sum_{k=1}^{N-1}\lambda\left(S_{p(N)}(f)=S_{p(k)}(f)\right)\\
&\leq \sum_{k=1}^{N-1}\frac{c}{N+(N-k)^2}\\
&\leq\sum_{k=1}^{N-1}\frac{c}{N+k^2}\\
&\leq c\int_1^N\frac{dx}{N+x^2}\leq \frac{\pi c}{\sqrt{N}}.
\end{align*} 
This implies that for all $N>M$,
\[
1-\frac{\pi c}{\sqrt{N}}\leq \lambda\left(A_N\right) \leq 1.
\]
To see \cref{item_a:R_n_estimate}, observe that,
\[
\underset{n\to\infty}{\lim}\,\,\frac{\mathbb{E}\left(\left|\bR_n^{(p)}\right|\right)}{n}=\underset{n\to\infty}{\lim}\,\,\frac{1}{n}\left(1+\sum_{N=2}^n\lambda\left(A_N\right)\right)= 1. 
\] 
\Cref{item_b:R_n_estimate} follows from
\begin{align*}
\mathbb{E}\left(\left|\bR_n^{(p)}\right|\right)&\geq \sum_{N=M}^n\left(1-\frac{\pi c}{\sqrt{N}}\right),\\
&\geq n-K\sqrt{n},
\end{align*}
where $K>0$, is a constant.
To see \cref{item_c:R_n_estimate} note that by \cref{item_b:R_n_estimate} and the fact that for all $y \in \T$, $\left|R_n^{(p)}(y)\right|\leq n$,
\begin{align*}
\mathrm{Var}\left(\left|\bR_n^{(p)}\right|\right)&=\mathbb{E}\left(\left|\bR_n^{(p)}\right|^2\right)-\left(\mathbb{E}\left(\left|\bR_n^{(p)}\right|\right)\right)^2,\\
&\leq n^2-\left(n-C\sqrt{n}\right)^2,\\
&\leq 2Cn^{\frac{3}{2}}.
\end{align*}
\end{proof} Fix $y\in\T$. Let $R^{(p)}(y) \subset \Z^2$ be as defined in \eqref{eq:R_p_y}. Observe that for every $z \in R^{(p)}(y)$ there exists a minimal $n\in\N$, such that $S_{p(n)}(f)(y) = z$. We define
\begin{equation}\label{eq:k_p_y}
   K^{(p)}(y):=\left\{n\in\N:\ \forall \ 1\leq m < n,\ S_{p(m)}(f)(y)\neq S_{p(n)}(f)(y),\, S_{p(n)}(f)(y) \neq (0,0)\right\}\subset \N.   
\end{equation} to be the collection of all such points. The results on the range give the following.
\begin{cor}\label{cor:k_p_y_banach} Let $p:\Z\to\Z$ be a polynomial with a positive leading coefficient and $\deg(p)\geq 2$, then for Lebesgue almost every $y \in \T$, $K^{(p)}(y)$ has Banach density one.
\end{cor}
\begin{proof}
For every $z\in \bR_n^{(p)}(y)\setminus\{(0,0)\}$ there exists a unique $k\in K^{(p)}(y)\cap [0,n]$ such that $S_{p(k)}(f)(y) = z$. Hence it follows,  
\[
\left|\left|\bR_n^{(p)}(y)\right|-\left|K^{(p)}(y)\cap[0,n]\right|\right|\leq 1.
\]
The claim follows from this and Proposition \ref{prop:Range_of_cocycle_along_p}. 
\end{proof}
Choose two increasing sequences of natural numbers $\left(N_k\right)_{k=1}^\infty$ and $\left(M_k\right)_{k=1}^\infty$ such that:
\begin{itemize}
 \item For every $k\in\N$, $N_k< M_k<N_{k+1}$.
 \item $\lim_{k\to\infty}\frac{M_k}{N_k}=\lim_{k\to\infty}\frac{N_{k+1}}{M_k}=\infty.$
\end{itemize}
We define $J \subset \N$ via,
\begin{equation}\label{eq:def_J}
    J:=\N\cap \uplus_{k=1}^\infty \left(N_k,M_k\right]
\end{equation} It is immediate that, 
\[
\lim_{k\to\infty}\frac{J\cap [0,M_k]}{M_k}=1\ \text{and}\ \lim_{k\to\infty}\frac{J\cap [0,N_k]}{N_k}=0.
\] In particular $J$ is of lower Banach density $0$ and of upper Banach density $1$. 
Let $p_1, p_2 : \Z \rightarrow \Z$ be polynomials with positive leading coefficients and $\deg(p_1), \deg(p_2) \geq 2$. Let $D\subset \T$ be the set of all points $y\in\T$ such that
\[
\lim_{n\to\infty}\frac{\left|\bR_n^{(p_1)}(y)\right|}{n}=\lim_{n\to\infty}\frac{\left|\bR_n^{(p_2)}(y)\right|}{n}=1.
\]
By \cref{prop:Range_of_cocycle_along_p}, $D$ has full measure. 
For $y \in D$, we define
\begin{equation}\label{eq:def_curly_K_y}
\mathcal{K}_y=K^{(p_1)}(y)\cap K^{(p_2)}(y)\cap J.
\end{equation} Observe that by \cref{cor:k_p_y_banach} and its proof, for all $y\in D$ and $j\in \{1,2\}$, 
\[
S(j,y) : = \left\{S_{p_j(n)}(f)(y):\ n\in \mathcal{K}_y\right\}\subset \Z^2,
\] is infinite and contains distinct terms. For  $y \in D$ and $j\in \{1,2\}$, the complement of $S(j,y)$ is also infinite. To see this observe that for $y \in D$, we have
\[
\bigcup_{k=1}^\infty \left(\bR_{N_{k+1}}^{(p_j)}(y)\setminus \bR_{M_k}^{(p_j)}(y)\right)\subset \Z^2\setminus S(j,y). 
\] Since for all $y \in D$, we have
\[
\lim_{k\to\infty}\frac{\left|\bR_{N_{k+1}}^{(p_j)}(y)\setminus \bR_{M_k}^{(p)}(y)\right|}{N_{k+1}}=1,
\] it follows that $\Z^2\setminus S(j,y)$ is infinite for $j\in \{1,2\}$. 

For $j \in \{1,2\}$, the mapping $\Theta_j: D\to 2^{\Z^2}$, defined by 
\begin{equation}\label{eq:theta_j}
    \Theta_j(y):=S(j,y) 
\end{equation} is measurable because of the following claim. 
\begin{lemma}\label{lem:O_measurable}
The map $\mathcal{O}: \T \rightarrow 2^{J}$ given by $\mathcal{O} (y) := \mathcal{K}_y$, is Borel.
\end{lemma}

\begin{proof}
Enumerate $J$ in increasing order and consider $2^{J}$ endowed with product topology based on the enumeration. We define for $n\in J$, 
\[
B_n:=\bigcap_{i\in\{1,2\}}\left\{y\in \T:\ \forall 1\leq j<n,\ S_{p_i(n)}(f)(y)-S_{p_i(j)}(f)(y)\neq (0,0)\ \text{and}\ S_{p_i(n)}(f)(y)\neq (0,0)\right\}.
\]
Since $f$ is measurable, $B_n\in\B(\T)$. For $A\in \B(\T)$ and $\epsilon\in\{0,1\}$, define
\[
A^\epsilon:=\begin{cases}
A, &\epsilon=1,\\
\T\setminus A,&\epsilon=0.
\end{cases}
\]
 For $F\subset J$ finite and $z\in\{0,1\}^F$, let $[z]_{F}$ be the corresponding cylinder set given by 
\[
[z]_{F} = \{\eta \in \{0,1\}^J: \eta(i) = z(i),\,\, \forall i\in F\}.
\] The map $\mathcal{O}$ is measurable because for all $F\subset J$ finite and $z\in\{0,1\}^F$, \[
\mathcal{O}^{-1}\left([z]_F\right)=\bigcap_{n\in F}\left(B_n\right)^{z_n}\in\B(\T).
\] \end{proof} In what follows, for $y \in D$, we will define a permutation $\pi_y: \Z^2 \rightarrow \Z^2$, such that $\pi_y$ maps
\begin{equation}\label{eq:pi_y_maps}
    (0,0) \mapsto (0,0), \hspace{4mm} \mathrm{and} \hspace{4mm} S_{p_2(n)} (f) (y) \mapsto S_{p_1(n)} (f) (y),\, \forall n \in \mathcal{K}_y.
\end{equation} To this effect, for $y \in D$ we fix an enumeration of $\mathcal{K}_y = \{k_1(y)<k_2(y)< \ldots\}$. For ease of notation, when $y$ is known, we will denote $\mathcal{K}_y = \{k_1 < k_2 < \ldots\}$. Thus for $y \in D$ and $j\in \{1,2\}$, we enumerate $S(j,y) = \{S_{p_j(k_i)} (f) (y) \}_{i=1}^{\infty} \subset \Z^2$. For $y \in D$ and $j \in \{1,2\}$ we set 
\begin{equation}\label{eq:L_j_y}
    L(j,y) = \Z^2\setminus (S(j,y) \cup \{(0,0)\}) = \Z^2 \setminus (\{S_{p_j(k_i)} (f) (y) \}_{i=1}^{\infty} \cup \{(0,0)\}) \subset \Z^2.
\end{equation} $L(j,y)$ is also infinite as discussed above. Let $L(j,y) := \{\ell(j,y)_1, \ell(j,y)_2,\ldots\} \subset \Z^2$ be an enumeration of the set $L(j,y)$. For $y \in D$, $j \in \{1,2\}$, we have partition of $\Z^2$ of the form
\[
\Z^2 = \{(0,0)\} \cup S(j,y) \cup L(j,y) = \{(0,0)\} \cup \{S_{p_j(k_i)} (f) (y) \}_{i=1}^{\infty} \cup \{\ell(j,y)_i\}_{i=1}^{\infty}.
\] For $y\in D$, $j \in \{1,2\}$, let $\pi_{p_j,y}: \Z \rightarrow \Z^2$, be a bijective map given by,
\begin{equation}\label{eq:pi_p_j_y}
    \pi_{p_j,y}(i):=\begin{cases}
(0,0), &\ \text{for}\,\, i = 0;\\
S_{p_j(k_i)} (f)(y), &\ \text{for}\,\, i \geq 1;\\ 
\ell(j,y)_{-i}. &\ \text{for}\,\, i \leq -1.
\end{cases}
\end{equation} For $y \in D$, we define a map, $\pi_y: \Z^2 \rightarrow \Z^2$, by 
\begin{equation}\label{eq:pi_y}
    \pi_{y} = \pi_{p_1,y} \circ \pi^{-1}_{p_2,y}.
\end{equation} Note that $\pi_y$ is a permutation of $\Z^2$ and it satisfies \eqref{eq:pi_y_maps} as needed.

Let $a:\N\to\Z^2$ and define a metric on the permutations of $\Z^2$ by for all bijections $\pi,\eta:\Z^2\to\Z^2$, 
\[
d(\pi,\eta):=2^{-\inf\left\{n\in\ \N:\ \pi\left(a_n\right)\neq \eta\left(a_n\right)\right\}}+2^{-\inf\left\{n\in\ \N:\ \pi^{-1}\left(a_n\right)\neq \eta^{-1}\left(a_n\right)\right\}}.
\]
The space of permutations of $\Z^2$ with this metric is a Polish space. Below we argue that the for Lebesgue almost every $y \in \T$, $y\mapsto \pi_y$ is a measurable map to the permutations of $\Z^2$. 

\begin{lemma}\label{lem:measurability_of_pi}
The map $y\mapsto \pi_y$ is a measurable map from $D \subset \T$ to the permutations of $\Z^2$.    
\end{lemma}

\begin{proof} To see this note that \cref{lem:O_measurable}, implies that the map $\Theta_j$ defined in \eqref{eq:theta_j} is measurable for $j \in \{1,2\}$. By definition of $L(j,y)$, (see \eqref{eq:L_j_y}) it follows that the map $y \mapsto L(j,y)$ from $D$ to $2^{\Z^2}$ is measurable for $j \in \{1,2\}$. This together with the definition of $\pi_{p_j,y}$ and $\pi_y$ implies that $y\mapsto \pi_y$ is a measurable map from $D \subset \T$ to the permutations of $\Z^2$. \end{proof} 

Recall that $X = \T \times \Sigma$ where $\T = \R /\Z$ and $\Sigma = \{0,1\}^{\Z^2}$. For $y\in D$, let $\pi_y$ be as above, we set $\Psi_{\pi_y}: \Sigma \rightarrow \Sigma$ to be the map
\begin{equation}\label{eq:Psi_pi_y}
    \Psi_{\pi_y} (\omega)(i,j):=\begin{cases}
\omega(0,0), &\ \text{for}\,\, (i,j) = (0,0);\\
1-\omega(\pi_y(i,j)) = 1-\omega(S_{p_1(n)} (f)(y)), &\ \text{for}\, (i,j) = S_{p_2(n)} (f)(y), n \in \mathcal{K}_y ;\\
\omega(\pi_y(i,j)), &\  \text{otherwise}.
\end{cases}
\end{equation} Let $Q: X \rightarrow X$ be given by,
\begin{equation}\label{eq:def_R}
    Q(y,\omega) :=\begin{cases}
(y, \Psi_{\pi_y} \omega ), &\ \text{for}\,\, y \in D;\\
(y,\omega), &\ \text{for}\,\, y \in \T\setminus D.
\end{cases}
\end{equation} We finally define, $S : X \rightarrow X$ as, 
\begin{equation}\label{def:S}
    S : = Q^{-1} \circ T \circ Q.
\end{equation} Observe that for $n \in \N$, $y \in D \cap R_{\alpha}^{-n} D$ and $\omega \in \Sigma$, we have
\begin{equation}\label{eq:S^n}
    S^n(y, \omega) = (y+n\alpha, (\Psi^{-1}_{\pi_{y + n \alpha}} \circ \sigma_{S_n(f)(y)} \circ \Psi_{\pi_y}) (\omega )).
\end{equation}
We argue that $(X,\mathcal{X},\mu,S)$ is an ergodic measure preserving system with zero entropy since the map $Q$ defined in \eqref{eq:def_R} is an invertible measure preserving transformation. 

In order to prove that the map $Q: X \rightarrow X$, defined in \eqref{eq:def_R} is an invertible measure preserving transformation we will rely on Souslin's Theorem (see \cite[Theorem 14.12]{Kechris_1995}) which says that if $f: A \rightarrow B$ is a Borel bijection, then $f$ is a Borel isomorphism (in other words $f^{-1}$ is also a Borel map). 

\begin{prop}\label{prop:R_mps} $(X,\mathcal{X}, \mu, Q)$ is an invertible measure preserving system. 
\end{prop}

\begin{proof} 
The map $Q$ is a bijection, hence by Souslin's theorem if $Q^{-1}$ is measurable, then  $Q$ is a Borel isomorpshim. Fix $y\in\T$. For every $F\subset \Z^2$ and $z\in \{0,1\}^F$, define $z(y)\in\{0,1\}^{\pi_y(F)}$ as follows, if $y\in D$ then, 
\[
z(y)_{\pi_y(j)}:=\begin{cases}
1-z_{\pi_y(j)},& j\in S(1,y)\cap F\\
z_{\pi_y(j)},& j\in F\setminus S(1,y),
\end{cases}
\]
and if $y\notin D$ then $z(y)=z$ and $\pi_y$ is the identity map on $\Z^2$. 
Now for every set of the form $B\times [z]_F\in\B(\T)\times\B(\Sigma)$,
\[
Q\left(B\times [z]_F\right):=\left\{(y,\omega):\ y\in B \ \text{and}\ \omega|_{\pi_y(F)}=z(y)\right\}. 
\] By \cref{lem:measurability_of_pi}, and the measurability of $y\mapsto S(1,y)$, it follows that $Q\left(B\times [z]_F\right)\in\B(\T)\times\B(\Sigma)$. Hence the map $Q^{-1}$ is measurable. The map $Q$ is measure preserving since for every $B\in\B(\T)$, $F\subset \Z^2$ finite and $z\in\{0,1\}^F$,
\begin{align*}
\mu\left(Q\left(B\times [z]_F\right)\right)&=\int_B \nu\left([z(y)]_F\right) d\lambda(y),\ \ \text{by Fubini theorem}\\
&=\int_B 2^{-|F|}d\lambda(y)=\mu\left(B\times [z]_F\right).
\end{align*} \end{proof}

\begin{cor} $(X,\mathcal{X},\mu,S)$ is an ergodic measure preserving system with $h_{\mu} (X,S) = 0$.
\end{cor}

\begin{proof} This follows from \Cref{prop:R_mps}, \Cref{prop:T_ergodic} and \Cref{prop:entropy_T=0}. \end{proof}
\Cref{thm:Main} directly follows from the following Lemma. 

\begin{lemma}\label{lem:divergence} Let $A = D \times [0]_{(0,0)}$, then for the two increasing sequence of natural numbers   $\left(N_k\right)_{k=1}^\infty$ and $\left(M_k\right)_{k=1}^\infty$ (as in the definition of $J$, see \eqref{eq:def_J}), we have
\[
\underset{k\rightarrow \infty}{\mathrm{lim}}\frac{1}{N_k}\sum_{n=0}^{N_k-1} \mu\left(T^{-p_1(n)}A\cap S^{-p_2(n)}A\right) \geq \dfrac{1}{4},
\] and \[
\underset{k\rightarrow \infty}{\mathrm{lim}}\frac{1}{M_k}\sum_{n=0}^{M_k-1} \mu\left(T^{-p_1(n)}A\cap S^{-p_2(n)}A\right) = 0. 
\]
\end{lemma}
\begin{proof} Note that $(y,\omega) \in D \times \Sigma \cap T^{-p_1(n)}A\cap S^{-p_2(n)}A$ if and only if $y \in D$, $T^{p_1(n)} (y,\omega) \in D \times [0]_{(0,0)}$ and $S^{p_2(n)} (y,\omega) \in D \times [0]_{(0,0)}$. For $n \in \N$ we set,
\begin{align*}
B(n) & : = D \times \Sigma \cap T^{-p_1(n)}A\cap S^{-p_2(n)}A \\
& = \{(y,\omega) \in X : y \in D,\,\textrm{and}\, T^{p_1(n)} (y,\omega), S^{p_2(n)} (y,\omega) \in D \times [0]_{(0,0)}\}.
\end{align*} 

Using the definitions of $T^{p_1(n)}$ and $S^{p_2(n)}$ (see \eqref{eq:T^n} and \eqref{eq:S^n}), we get
\[
B(n) = \{(y,\omega) \in X: y \in D, \sigma_{S_{p_1(n)}(f)(y)} (\omega) (0,0) = 0,\,\textrm{and}\, (\Psi^{-1}_{\pi_{y + n \alpha}} \circ \sigma_{S_{p_2(n)}(f)(y)} \circ \Psi_{\pi_y}) (\omega ) (0,0) = 0 \}.
\] Observe that for $(\tilde{y},\tilde{\omega}) \in D \times \Sigma$, we have $(\Psi^{-1}_{\pi_{\tilde{y}}}\tilde{\omega}) (0,0) = \tilde{\omega} (0,0)$, hence
\begin{align}
B(n) &= \{(y,\omega) \in X: y \in D, \sigma_{S_{p_1(n)}(f)(y)} (\omega) (0,0) = 0,\,\textrm{and}\, ( \sigma_{S_{p_2(n)}(f)(y)} \circ \Psi_{\pi_y}) (\omega ) (0,0) = 0 \} \nonumber\\
&= \{(y,\omega) \in X: y \in D,  (\omega) (S_{p_1(n)}(f)(y)) = 0,\,\textrm{and}\, (\Psi_{\pi_y} \omega ) (S_{p_2(n)}(f)(y)) = 0 \}.\label{eq:def_set_B}
\end{align} 
If $n\notin \mathcal{K}_y$, then either $\pi_y \left(S_{p_2(n)}(f)(y)\right)=S_{p_1(n)}(f)(y)$ and 
\[
\Psi_{\pi_y}(S_{p_2(n)}(f)(y))=\omega(S_{p_1(n)}(f)(y))
\]
or $\pi_y \left(S_{p_2(n)}(f)(y)\right)\neq S_{p_1(n)}(f)(y)$ and then $\Psi_{\pi_y}(S_{p_2(n)}(f)(y))$ and $\omega(S_{p_1(n)}(f)(y))$ are independent. By this and the definition of $\Psi_{\pi_y}$ we have for all $y\in D$ and $n\notin \mathcal{K}_y$
\[
\nu\left(\omega\in\Sigma:\  \sigma_{S_{p_1(n)}(f)(y)} (\omega) (0,0) = 0,\,\textrm{and}\, ( \sigma_{S_{p_2(n)}(f)(y)} \circ \Psi_{\pi_y}) (\omega ) (0,0) = 0\right)\geq \frac{1}{4}.
\]
This together with Fubini's Theorem implies that 
\[
\text{for all}\ n\notin \mathcal{K}_y,\ \mu\left(T^{-p_1(n)}A\cap S^{-p_2(n)}A\right)=\mu(B(n))\geq \frac{1}{4}.
\]

Now we calculate,
\begin{align*}
&\frac{1}{N_k}\sum_{n=0}^{N_k-1} \mu\left(T^{-p_1(n)}A\cap S^{-p_2(n)}A\right)\\
&\geq \frac{1}{N_k} \sum_{n=M_{k-1}+1}^{N_k-1} \mu\left(T^{-p_1(n)}A\cap S^{-p_2(n)}A\right).
\end{align*} Note that $J \cap [M_{k-1}, N_{k}-1] = \emptyset$, hence if $n \in [M_{k-1}+1, N_{k}-1]$, then $n \notin \mathcal{K}_y$. Hence it follows that
\[
\underset{k\rightarrow \infty}{\mathrm{lim}}\frac{1}{N_k}\sum_{n=0}^{N_k-1} \mu\left(T^{-p_1(n)}A\cap S^{-p_2(n)}A\right) \geq \underset{k\rightarrow \infty}{\mathrm{lim}} \frac{1}{N_k} \sum_{n=M_{k-1}+1}^{N_k-1} \frac{1}{4} = \frac{1}{4}.
\] Similarly, we calculate
\begin{align*}
    &\frac{1}{M_k}\sum_{n=0}^{M_k-1} \mu\left(T^{-p_1(n)}A\cap S^{-p_2(n)}A\right)\\
    & = \int_{X} \frac{1}{M_k} \sum_{n=0}^{M_k-1} \big(1_A \circ T^{p_1(n)} (y,\omega)\big) \circ \big(1_A \circ S^{p_2(n)} (y,\omega)\big) d\mu\\
    & \leq \dfrac{N_k}{M_k} + \int_{X} \frac{1}{M_k} \sum_{n\, \in \, \mathcal{K}_y \cap [N_k, M_k-1]} \big(1_A \circ T^{p_1(n)}(y,\omega)\big) \circ \big(1_A \circ S^{p_2(n)}) (y,\omega)\big) d\mu\\
    & + \int_{X} \frac{1}{M_k} \sum_{n\,\in\,[N_k, M_k-1] \setminus \mathcal{K}_y} \big(1_A \circ T^{p_1(n)}(y,\omega)\big) \circ \big(1_A \circ S^{p_2(n)}) (y,\omega)\big) d\mu.
\end{align*} Observe that \eqref{eq:Psi_pi_y} implies that for $(y,\omega) \in D \times \Sigma$ and $n\in \mathcal{K}_y$,
\begin{equation*}\label{eq:psi_omega}
    (\Psi_{\pi_y} \omega ) (S_{p_2(n)}(f)(y))=
1- \omega(S_{p_1(n)}(f)(y)).
\end{equation*}
Taking into consideration \eqref{eq:def_set_B} we conclude that for all  $(y,\omega) \in D \times \Sigma$ and $n\in \mathcal{K}_y$, 
\[
\big(1_A \circ T^{p_1(n)}(y,\omega)\big) \circ \big(1_A \circ S^{p_2(n)}) (y,\omega)\big)  = 0. 
\] Hence, as $\frac{N_k}{M_k} \xrightarrow[]{k\to\infty}0$, 
\begin{align*}
    &\underset{k\rightarrow \infty}{\mathrm{lim}}\frac{1}{M_k}\sum_{n=0}^{M_k-1} \mu\left(T^{-p_1(n)}A\cap S^{-p_2(n)}A\right)\\
    &= \underset{k\rightarrow \infty}{\mathrm{lim}} \int_{X} \frac{1}{M_k} \sum_{n\,\in\,[N_k, M_k-1] \setminus \mathcal{K}_y} \big(1_A \circ T^{p_1(n)}(y,\omega)\big) \circ \big(1_A \circ S^{p_2(n)}) (y,\omega)\big) d\mu.
\end{align*} For $k \in \N$, we set
\[
g_k : = \frac{1}{M_k} \sum_{n\,\in\,[N_k, M_k-1] \setminus \mathcal{K}_y} \big(1_A \circ T^{p_1(n)}(y,\omega)\big) \circ \big(1_A \circ S^{p_2(n)}) (y,\omega)\big) . 
\] It follows from the definition of $\mathcal{K}_y$ (see \eqref{eq:def_curly_K_y}), that the map $\T \ni y \mapsto [N_k, M_{k}-1] \cap \mathcal{K}_y$ is measurable. Hence the map, $\T \ni y \mapsto [N_k, M_{k}-1] \setminus \mathcal{K}_y$ is measurable. This implies that for every $k \in \N$, the map $g_k$ is measurable. Also for every $k \in \N$,\[
g_k (y,\omega) \leq \dfrac{\left|[N_k, M_{k}-1] \setminus \mathcal{K}_y \right|}{M_k} \leq 1. 
\] From \Cref{cor:k_p_y_banach} and the definition of $\mathcal{K}_y$, it follows that for $y \in D$,
\[
\underset{k\rightarrow \infty}{\mathrm{lim}} \dfrac{\left|[N_k, M_{k}-1] \cap \mathcal{K}_y \right|}{M_k} = 1
\] This implies that $g_k \xrightarrow[]{k\to\infty}0$. Thus by using Bounded Convergence Theorem,
\[
\underset{k\rightarrow \infty}{\mathrm{lim}}\frac{1}{M_k}\sum_{n=0}^{M_k-1} \mu\left(T^{-p_1(n)}A\cap S^{-p_2(n)}A\right) = \underset{k\rightarrow \infty}{\mathrm{lim}} \int_{X} g_k = 0.
\]
\end{proof}

\subsection{Proof of Theorem \ref{thm:recurrence counterexample}}
\label{sub: recurrence proof}
The proof of \cref{thm:recurrence counterexample} is by modifying the construction of $S$. Let $f:\T\to\Z^2$ be a function as in Subsection \ref{sec:construction_T} and for an integer polynomial  $p$, write  
\[
\CC_p:=\left\{y\in\T:\ \forall n\neq j,\ n,j\in \N, S_{p(n)}(f)(y)\neq S_{p(j)}(f)(y),\ \text{and}\ S_{p(n)}(f)(y)\neq (0,0)\right\}.
\]
\begin{prop}\label{prop: no returens along special polynomials}
There exists $M>0$, such that if $p(n)=Ln^d$ with $\N\ni L>M$ and $d\geq 3$,
then $\lambda\left(\CC_p\right)>0$. 
\end{prop}
We begin with the following variant of \cref{claim:simple_bound}. 
\begin{claim}\label{claim:simple_bound2}
For all $d\geq 3$ and $n,k\in \N$ such that $n>k$,
\[
n^d-k^d\geq n^2+(n-k)^3.
\]
\end{claim}
\begin{proof}
Since for all $d\geq 3$ and integers $k<n$, 
\[
n^d-k^d\geq n^3-k^3,
\]
it remains to prove the bound for $d=3$. The bound now follows from,
\begin{align*}
 n^3-k^3&=\left(n-k\right)\left(n^2+kn+k^2\right)\\
 &=\left(n-k\right)\left((n-k)^2+3kn\right)\ \ \ \ \text{since}\ 3k(n-k)>n\\
&\geq (n-k)^3+n^2.
\end{align*} \end{proof}

\begin{proof}[Proof of \cref{prop: no returens along special polynomials}] Let $p(n)=Ln^d$ with $L\in\N$ and $d\geq 3$. For $n, j \in \N$, $n\neq j$ define, 
\[
\CC_p (n,j) = \{y \in \T: n,j \in \N , S_{p(n)}(f)(y)\neq S_{p(j)}(f)(y)\}
\] and
\[
\CC_p(n,0) = \{ y \in \T : S_{p(n)}(f)(y) = (0,0)\}.
\] Observe that,
\[
\CC_p = \T \setminus \left(\bigcup_{n=1}^{\infty} \CC_p(n,0) \cup \bigcup_{n=1}^{\infty} \bigcup_{m=n+1}^{\infty} \CC_p(n,m)\right).
\] As in the proof of \cref{lem:expectation_and_variance_bound_on_range}, using  the local central limit theorem (\cref{thm:LCLT}) for $d=2$, there exists $\beta>0$, such that for every $n, m \in \N$, $m > n$, we have
\begin{align*}
\lambda(C_p(n,m)) = \lambda\left(S_{p(m)}(f)-S_{p(n)}(f)=(0,0)\right)&=\lambda\left(S_{p(m)-p(n)}(f)\circ R_\alpha^{p(n)}=(0,0)\right)\\
&=\lambda\left(S_{p(m)-p(n)}(f)=(0,0)\right)\leq \frac{\beta}{p(m)-p(n)}.
\end{align*} By \cref{claim:simple_bound2}, we have 
\[
p(m)-p(n) = L (m^d - n^d) > L (m^2 + (m-n)^3). 
\] Thus,
\begin{align*}
\lambda \left(\bigcup_{n=1}^{\infty} \bigcup_{m=n+1}^{\infty} \CC_p(n,m) \right) \leq \sum_{n=1}^{\infty} \sum_{m=n+1}^{\infty} \lambda \left(\CC_p(n,m) \right) &< \sum_{n=1}^{\infty} \sum_{m=n+1}^{\infty}\frac{\beta}{L (m^2 + (m-n)^3)}\\
&< \dfrac{\beta}{L} \sum_{n=1}^{\infty} \sum_{t=1}^{\infty} \dfrac{1}{n^2 + t^3}.
\end{align*} Since $\displaystyle \sum_{n=1}^{\infty} \sum_{t=1}^{\infty} \dfrac{1}{n^2 + t^3}$ is finite, we see that for all large enough $L\in\N$, 
\[
\lambda \left(\bigcup_{n=1}^{\infty} \bigcup_{m=n+1}^{\infty} \CC_p(n,m) \right)<\frac{1}{2}.
\]Similarly, as $\sum_{n=1}^\infty \frac{1}{p(n)}$ is finite, a similar argument shows that for all large enough $L \in \N$, we have $\lambda \left( \bigcup_{n=1}^{\infty} \CC_p(n,0) \right) < 1/2$. Hence, $\lambda(\CC_p)>0$ as needed. 
\end{proof}

Recall that $X = \T \times \Sigma$ where $\T = \R /\Z$ and $\Sigma = \{0,1\}^{\Z^2}$. For $y\in \CC_p$, we set $\Psi_y: \Sigma \rightarrow \Sigma$ to be the map
\begin{equation}\label{eq:Psi_y}
    \Psi_y (\omega)(i,j):=\begin{cases}
1-\omega(S_{p(n)} (f)(y)), &\ \text{for}\, (i,j) = S_{p(n)} (f)(y),\,\, \forall n \in \N \\
\omega(i,j), &\  \text{otherwise}.
\end{cases}
\end{equation} Let $V: X \rightarrow X$ be given by,
\begin{equation}\label{eq:def_R_tilde}
   V(y,\omega) :=\begin{cases}
(y, \Psi_y (\omega) ), &\ \text{for}\,\, y \in \CC_p;\\
(y,\omega), &\ \text{for}\,\, y \in \T\setminus \CC_p.
\end{cases}
\end{equation} We finally define, $S : X \rightarrow X$ as, 
\begin{equation}\label{def:S_tilde}
   S : = V^{-1} \circ T \circ V.
\end{equation} where $T : X \rightarrow X$ is the skew product of $R_\alpha$ and $f$ as defined in Subsection \ref{sec:construction_T}. Observe that for $n \in \N$, $y \in \CC_p \cap R_{\alpha}^{-n} \CC_p$ and $\omega \in \Sigma$, we have
\begin{equation}
    S^n(y, \omega) = (y+n\alpha, (\Psi^{-1}_{y + n \alpha} \circ \sigma_{S_n(f)(y)} \circ \Psi_y) (\omega )).
\end{equation}
It follows from an argument similar to \cref{prop:R_mps} that $(X,\mathcal{X}, \mu, V)$ is an invertible measure preserving system. Hence it follows that $(X,\mathcal{X},\mu,S)$ is isomorphic to $T$, consequently $S$ is an ergodic measure preserving system with $h_{\mu} (X,S) = 0$. 

\begin{proof}[Proof of \cref{thm:recurrence counterexample}] Let $M>0$ as in \cref{prop: no returens along special polynomials} and $p$ a polynomial of the form $p(n)=Ln^d$ with $L>M$ and $d\geq 3$. Define $A = \CC_p \times [0]_{(0,0)}$ and note that by \cref{prop: no returens along special polynomials},
\[
\mu(A)=\frac{1}{2}\lambda\left(\CC_p\right)>0.
\]We claim that for all $n\in\N$,
\[
\mu\left(A \cap T^{-p(n)}A\cap S^{-p(n)}A\right)=0.
\]
Indeed, if $(y,\omega)\in A \cap T^{-p(n)}A\cap S^{-p(n)}A$, then $y,y+n\alpha\in\CC_p$ and 
\[
\omega\left(S_{p(n)}(f)(y)\right)=\Psi^{-1}_{y + n \alpha} \circ \sigma_{S_{p(n)}(f)(y)} \circ \Psi_y (\omega )(0,0).
\] 
For all $y\in\CC_p$, $(0,0)\notin \left\{S_{p(n)}(f)(y)\right\}_{n=1}^\infty$. Consequently, for all $y\in\T$, such that $y,y+n\alpha\in\CC_p$,
\begin{align*}
\Psi^{-1}_{y + n \alpha} \circ \sigma_{S_{p(n)}(f)(y)} \circ \Psi_y (\omega )(0,0)&=\Psi^{-1}_{y + n \alpha}\left(\sigma_{S_{p(n)}(f)(y)} \circ \Psi_y (\omega )\right)(0,0)\\
&=\sigma_{S_{p(n)}(f)(y)} \circ \Psi_y (\omega )(0,0)\\
&=\Psi_y (\omega )\left(S_{p(n)}(f)(y)\right)\\
&=1-\omega\left(S_{p(n)}(f)(y)\right)\neq \omega\left(S_{p(n)}(f)(y)\right).
\end{align*}
This implies that for all $n\in\N$, 
\[
\mu\left(A \cap T^{-p(n)}A\cap S^{-p(n)}A\right)=\mu\left(\emptyset\right)=0.
\]
\end{proof}

\appendix
\section{}\label{App_1} The following is a multi-dimensional version of \cite[Proposition 18]{MR4374685}. Its proof is similar to the $1$-dimensional case. Recall that for $x\in\R^D$, $\|x\|$ denotes the Euclidean norm of $x$. 

\begin{prop}\label{prop:bla}
Suppose that for each $n\in\N$, $X_n=Y_n+Z_n$ where $Y_n,Z_n$ are independent $\Z^D$ valued random variables and $\mathbb{E}\left(\left\|Z_n\right\|^2\right)=O\left(\frac{n}{\sqrt{\log(n)}}\right)$ and $\sigma>0$. If 
\begin{equation}\label{eq:LLT_for_Y}
\sup_{x\in \Z^D}\left|n^{D/2}\mathbb{P}\left(Y_n=x\right)-\frac{e^{-\frac{\|x\|^2}{2n}}}{\left(2\pi \sigma^2\right)^{D/2}}\right|=o(1)
\end{equation}
then 
\[
\sup_{x\in \Z^D}\left|n^{D/2}\mathbb{P}\left(X_n=x\right)-\frac{e^{-\frac{\|x\|^2}{2n}}}{\left(2\pi \sigma^2\right)^{D/2}}\right|=o(1)
\]
\end{prop}
\begin{proof}
Let $a(n):=\frac{\sqrt{n}}{\sqrt[8]{\log(n)}}$ and $C:=\sup_{n\in\N}\frac{\sqrt{\log(n)}}{n}\mathbb{E}\left(\left\|Z_n\right\|^2\right)<\infty.$ By Markov's inequality, for all $n\in\N$
\begin{equation*}
\mathbb{P}\left(\left\|Z_n\right\|>a(n)\right)\leq \frac{C}{\sqrt[4]{\log(n)}}.
\end{equation*}
Let $x\in\Z^D$. As $Y_n$ and $Z_n$ are independent we have
\[
\mathbb{P}\left(X_n=x\right)=\sum_{z\in\Z^D}\mathbb{P}\left(Y_n=x-z\right)\mathbb{P}\left(Z_n=z\right).
\]
We split the sum into $\|z\|>a(n)$ and $\|z\|\leq a(n)$. 

By \eqref{eq:LLT_for_Y}, there exists $\alpha>0$ such that for all $z\in \mathbb{Z}^D$, 
\[
\mathbb{P}\left(Y_n=x-z\right)\leq \frac{\alpha}{n^{D/2}}.
\]
We deduce that
\begin{align}\label{eq:case_z_geq_a_n}
\sum_{\|z\|>a(n)}\mathbb{P}\left(Y_n=x-z\right)\mathbb{P}\left(Z_n=z\right)&\leq \sum_{\|z\|>a(n)}\frac{\alpha}{n^{D/2}}\mathbb{P}\left(Z_n=z\right) \nonumber \\
&=\frac{\alpha}{n^{D/2}}\mathbb{P}\left(\left\|Z_n\right\|>a(n)\right)\leq \frac{C\alpha}{n^{D/2}\sqrt[4]{\log(n)}}.
\end{align}
We now turn to look at the sum when $\|z\|\leq a(n)$. Firstly if $\|x\|>\sqrt{n}\sqrt[9]{\log(n)}$, then for all $n>3$, $a(n)<\frac{\|x\|}{2}$. Consequently  for all $z\in\Z^d$ with $\|z\|\leq a(n)$,
\[
\|x-z\|\geq \|x\|-\|z\|>\frac{1}{2}\sqrt{n}\sqrt[9]{\log(n)}. 
\]
This and \eqref{eq:LLT_for_Y} imply that for all $n>3$ and uniformly on $\|z\|\leq a(n)$, 
\begin{align*}
\mathbb{P}\left(Y_n=x-z\right)&=\frac{e^{-\frac{\|x-z\|^2}{2n\sigma^2}}}{\left(2\pi n\sigma^2\right)^{D/2}}+o\left(\frac{1}{n^{D/2}}\right)\\
&\leq \frac{e^{-\frac{\left(\log(n)\right)^{2/9}}{4\sigma^2}}}{\left(2\pi n\sigma^2\right)^{D/2}}+o\left(\frac{1}{n^{D/2}}\right)=o\left(\frac{1}{n^{D/2}}\right).
\end{align*}
We conclude that for such $x$, 
\[
\sum_{\|z\|\leq a(n)}\mathbb{P}\left(Y_n=x-z\right)\mathbb{P}\left(Z_n=z\right)=o\left(\frac{1}{n^{D/2}}\right). 
\]
Taking in mind that for all $x$ with $\|x\|\geq \sqrt{n}\sqrt[9]{\log(n)}$, 
\[
e^{-\frac{\|x\|^2}{2n\sigma^2}}\leq e^{-\frac{(\log(n))^{2/9}}{2n\sigma^2}}=o\left(1\right),
\]
we have shown that
\begin{equation}\label{eq:LCLT_X_for_large_x}
\sup_{\|x\|>\sqrt{n}\sqrt[9]{\log(n)}}\left|n^{D/2}\mathbb{P}\left(X_n=x\right)-\frac{e^{-\frac{\|x\|^2}{2n}}}{\left(2\pi \sigma^2\right)^{D/2}}\right|=o(1).    
\end{equation}
When $\|x\|\leq \sqrt{n}\sqrt[9]{\log(n)}$, then for all $z$ with $\|z\|\leq a(n)$, we have 
\begin{align*}
\left|\|x-z\|^2-\|x\|^2\right|&\leq 2|\langle x,z\rangle|+\|z\|^2\\
&\leq \|x\|\|z\|+\|z\|^2\\
&\leq 2\left(\sqrt{n}\sqrt[9]{\log(n)}\right)a(n)= \frac{2n}{(\log(n))^{1/72}}.
\end{align*}
Consequently, for all $z$ with $\|z\|\leq a(n)$,
\begin{align*}
\mathbb{P}\left(Y_n=x-z\right)&=\frac{e^{-\frac{\|x-z\|^2}{2n\sigma^2}}}{\left(2\pi n\sigma^2\right)^{D/2}}+o\left(\frac{1}{n^{D/2}}\right)\\
&= e^{\pm\frac{2}{(\log(n))^{1/72}\sigma^2}}\frac{e^{-\frac{\|x\|^2}{2n\sigma^2}}}{\left(2\pi n\sigma^2\right)^{D/2}}+o\left(\frac{1}{n^{D/2}}\right)\\
&=\frac{e^{-\frac{\|x\|^2}{2n\sigma^2}}}{\left(2\pi n\sigma^2\right)^{D/2}}\left(1+o\left(\frac{1}{n^{D/2}}\right)\right).
\end{align*}
Here the term $1+o\left(n^{-D/2}\right)$ is uniform over all $x,z$ with $\|x\|\leq \sqrt{n}\sqrt[9]{\log(n)}$ and $\|z\|\leq a(n)$.  It follows that for such $x$, 
\begin{align*}
\sum_{\|z\|\leq a(n)}\mathbb{P}\left(Y_n=x-z\right)\mathbb{P}\left(Z_n=z\right)&=\frac{e^{-\frac{\|x\|^2}{2n\sigma^2}}}{\left(2\pi n\sigma^2\right)^{D/2}}\left(1+o\left(\frac{1}{n^{D/2}}\right)\right)\mathbb{P}\left(\|Z_n\|\leq a(n)\right)\\
&=\frac{e^{-\frac{\|x\|^2}{2n\sigma^2}}}{\left(2\pi n\sigma^2\right)^{D/2}}\left(1+o\left(\frac{1}{n^{D/2}}\right)\right).
\end{align*}
We conclude that 
\begin{equation}\label{eq:LCLT_X_for_small_x}
\sup_{\|x\|\leq \sqrt{n}\sqrt[9]{\log(n)}}\left|n^{D/2}\mathbb{P}\left(X_n=x\right)-\frac{e^{-\frac{\|x\|^2}{2n}}}{\left(2\pi n\sigma^2\right)^{D/2}}\right|=o(1).    
\end{equation}
The claim now follows from \eqref{eq:case_z_geq_a_n}, \eqref{eq:LCLT_X_for_large_x} and \eqref{eq:LCLT_X_for_small_x}. 
\end{proof} 

\section{Proof of Theorem \ref{thm:LCLT_for_U}}\label{Appendix B}

Recall that writing $I_n:=\left\{k\in2\N:\ p_{k}<n<d_k\right\}$ and $\mathsf{U}_n:=\sum_{k\in I_n}V_k$. For $n\in\N$, let
\[
\phi_n(t):=\mathbb{E}_m\left(\exp\left(it\mathsf{U}_n\right)\right)
\]In \cite{MR4374685} the following is proved. 

\begin{thm}\label{thm:steps_in_LCLT_for U_n}
\begin{enumerate}[label=(\alph*)]
\item\cite[Corollary 12]{MR4374685}\label{item_a:U_n_converge} $\frac{1}{\sqrt{n}}\mathsf{U}_n$ converges in distribution to a centered  normal random variable with variance\footnote{There is a typo in the definition of $\sigma^2$ in \cite[Corollary 12]{MR4374685}} $\sigma^2=2\left(\ln(2)\right)^2$. 
\item\cite[Lemma 14]{MR4374685}\label{item_b:phi_n_estimate} There exists $c>0$ such that for all $\sqrt[4]{n}\leq x\leq \pi\sqrt{n}$,
\[
\left|\phi_n\left(\frac{x}{\sqrt{n}}\right)\right|\leq \exp\left(-x\sqrt[4]{n}\right)\leq \exp\left(-d\sqrt{|x|}\right). 
\]
where $d=\frac{c}{\sqrt{\pi}}$. 
\item\cite[Lemma 15]{MR4374685}\label{item_c:phi_n_estimate} There exists $N\in\N$ and $L>0$ such that for all $n>N$ and $|x|\leq \sqrt[4]{n}$,
\[
\left|\phi_n\left(\frac{x}{\sqrt{n}}\right)\right|\leq \exp\left(-Lx^2\right). 
\]
\end{enumerate}
\end{thm}
For $n\in\N$, write $Z_n:=\mathsf{U}_n-\mathsf{U}(f,n)$ and set
\[
\varphi_n(t):=\mathbb{E}_m\left(\exp\left(itZ_n\right)\right)
\]
\begin{lemma}\label{lem:junk_disappear}
$Z_n$ converges in probability to $0$, and there exists $N\in\N$, such that for all $n>N$, we have
\[
\min_{|t|\leq \pi}\left|\varphi_n(t)\right|\geq \frac{1}{2}. 
\]
\end{lemma}
\begin{proof}
By \cref{rem:change_of_U_term}, 
\[
Z_n=1_{\left[\log(n-1)\in 2\N\right]}V_{\log(n-1)}. 
\]
When $\log(n-1)\in 2\N$, we have $p_{\log(n-1)}=n-1$ and $p_{\log(n-1)+1}=n$. Consequently, writing $k=\log(n-1)$, we have
\[
V_{\log(n-1)}=(n-1)\left(U^{n-1}f_{k}-U^{d_{k}+n-1}f_{k}\right)+n\left(U^{n}f_{k+1}-U^{d_{k}+n}f_{k+1}\right).
\]
In this case, since $V_{\log(n-1)}$ is a sum of square integrable zero mean random variables, we deduce that
\[
\left\|V_{\log(n-1)}\right\|_2^2\leq 2\left((n-1)^2\alpha_{k}^2+n^2\alpha_{k+1}^2\right)\leq \frac{8}{\log(n)},
\]
where the last inequality holds since 
\[
n\alpha_{k+1}\leq (n-1)\alpha_k<\frac{n-1}{(n-1)\sqrt{\log(n-1)}}.
\]
A simple argument using Markov's inequality shows that $Z_n$ converges to $0$ in probability (hence to $\delta_0$ in distribution). 

As $\mathbb{E}_m\left(Z_n\right)=0$, it follows from \cite[Theorem 3.3.8 and formula (3.3.3)]{Durrett_2010} that for all $t\in [-\pi,\pi]$,
\[
\varphi_n(t)=1-\frac{t^2\left\|Z_n\right\|_2^2}{2}\pm t^2\left\|Z_n\right\|_2^2.
\]
The second claim is now a consequence of the second moment estimate since
\[
\sup_{|t|\leq \pi} t^2\left\|Z_n\right\|_2^2\leq \frac{8\pi}{\log(n)}\xrightarrow[n\to\infty]{} 0.    
\]
\end{proof}

Finally let $\psi_n :\R\to \mathbb{C}$ be defined by
\[
\psi_n(t):=\mathbb{E}_m\left(\exp\left(it\mathsf{U}(f,n)\right)\right)
\]

\begin{thm}\label{thm:U_f_n_all_estimates} \begin{enumerate}[label=(\alph*)]
\item\label{item_a:U_f_n} $\frac{1}{\sqrt{n}}\mathsf{U}(f,n)$ converges in distribution to a centered  normal random variable with variance $\sigma^2=2\left(\ln 2\right)^2$. 
\item\label{item_b:psi_n_estimate} There exists $N\in\N$ and $c>0$ such that for all $n>N$ and $\sqrt[4]{n}\leq x\leq \pi\sqrt{n}$,
\[
\left|\psi_n\left(\frac{x}{\sqrt{n}}\right)\right|\leq 2\exp\left(-c\sqrt[4]{n}\right)\leq 2\exp\left(-d\sqrt{|x|}\right). 
\]
where $d=\frac{c}{\sqrt{\pi}}$. 
\item\label{item_c:psi_n_estimate} There exists $N\in\N$ and $L>0$ such that for all $n>N$ and $|x|\leq \sqrt[4]{n}$,
\[
\psi_n\left(\frac{x}{\sqrt{n}}\right)\leq 2\exp\left(-Lx^2\right). 
\]
\end{enumerate}
\end{thm}
\begin{proof}
Since $\mathsf{U}(f,n)=\mathsf{U}_n-Z_n$, we deduce \cref{item_a:U_f_n} from \cref{thm:steps_in_LCLT_for U_n}.\ref{item_a:U_n_converge} and \cref{lem:junk_disappear}. In order to prove \cref{item_b:psi_n_estimate} and \cref{item_c:psi_n_estimate}, note that $Z_n$ and $\mathsf{U}(f,n)$ are independent and $\mathsf{U}_n=\mathsf{U}(f,n)+Z_n$. This implies that  for all $|x|\leq \pi \sqrt{n}$, we have
\begin{equation}\label{eq:Yes!!!}
\phi_n\left(\frac{x}{\sqrt{n}}\right)=\psi_n\left(\frac{x}{\sqrt{n}}\right)\varphi_n\left(\frac{x}{\sqrt{n}}\right).
\end{equation}
By \cref{lem:junk_disappear} for all large $n$ we have
\begin{equation}\label{eq:Yes!!}
\sup_{|x|\leq \sqrt{n}\pi}\left|\varphi_n\left(\frac{x}{\sqrt{n}}\right)\right|\geq \frac{1}{2}.
\end{equation}
Items \ref{item_b:psi_n_estimate} and \ref{item_c:psi_n_estimate} follow from \eqref{eq:Yes!!!}, \eqref{eq:Yes!!}, \cref{thm:steps_in_LCLT_for U_n}.\ref{item_b:phi_n_estimate} and \cref{thm:steps_in_LCLT_for U_n}.\ref{item_c:phi_n_estimate}. 
\end{proof} Now we are ready to prove \cref{thm:LCLT_for_U}.

\begin{proof}[Proof of \cref{thm:LCLT_for_U}] Recall that $\psi_n(t)$ is the characteristic function of $\mathsf{U}(f,n)$. By Fourier inversion formula for $m\in \Z$ we get,
\[
\mu (\mathsf{U}(f,n) = m) = \dfrac{1}{2\pi} \int_{-\pi}^{\pi} \psi_n (t) e^{-itm} dt.
\] By change of variables $t = \dfrac{x}{\sqrt{n}}$ we get,
\[
\sqrt{n} \mu (\mathsf{U}(f,n) = m) = \dfrac{1}{2\pi} \int_{-\pi \, \sqrt{n}}^{\pi \, \sqrt{n}} \psi_n \bigg(\frac{x}{\sqrt{n}}\bigg) e^{-(ixm)/\sqrt{n}} dx.
\] Observe that,
\[
\dfrac{1}{\sqrt{2\pi\sigma^2}} e^{-m^2/2n \sigma^2} = \dfrac{1}{2\pi} \int_{\R} e^{-\sigma^2 x^2/2} \, e^{-(ixm)/\sqrt{n}} dx.
\] Hence we need to show
\[
\underset{m\in \Z}{\mathrm{sup}}\left\vert \dfrac{1}{2\pi} \int_{-\pi \, \sqrt{n}}^{\pi \, \sqrt{n}} \psi_n \bigg(\frac{x}{\sqrt{n}}\bigg) e^{-(ixm)/\sqrt{n}} dx - \dfrac{1}{2\pi} \int_{\R} e^{-\sigma^2 x^2/2} \, e^{-(ixm)/\sqrt{n}} dx  \right\vert \underset{n \rightarrow \infty}{\longrightarrow} 0.
\]Since, 
\[
\int_{|x| \geq \pi\,\sqrt{n}} e^{-\sigma^2 x^2/2} \, e^{-(ixm)/\sqrt{n}} dx \underset{n \rightarrow \infty}{\longrightarrow} 0,
\]
using triangular inequality, we need to show
\[
\int_{-\pi \, \sqrt{n}}^{\pi \, \sqrt{n}} \left\vert \psi_n \bigg(\frac{x}{\sqrt{n}}\bigg) - e^{-\sigma^2 x^2/2} \right\vert dx \underset{n \rightarrow \infty}{\longrightarrow} 0.
\]
By \cref{thm:U_f_n_all_estimates}.\ref{item_a:U_f_n}, $\frac{1}{\sqrt{n}} \mathsf{U}(f,n)$ converges in distribution to a normal law with $\sigma^2 = 2 (\mathrm{ln }\, 2)^2$. We use Levy's continuity theorem to deduce that that characteristic function converges pointwise. In other words for all $x \in \R$, 
\[
\Psi_n (x) : = \mathbf{1}_{[-\pi\,\sqrt{n}, \pi\,\sqrt{n}\,]} (x) \left\vert \psi_n \bigg(\frac{x}{\sqrt{n}}\bigg) - e^{-\sigma^2 x^2/2} \right\vert \underset{n \rightarrow \infty}{\longrightarrow} 0.
\] Note that by \cref{thm:U_f_n_all_estimates} items \ref{item_b:psi_n_estimate} and \ref{item_c:psi_n_estimate}, for large $n$ and $|x| \leq \pi\,\sqrt{n}$ the function $\Psi_n (x)$ is bounded by an integrable function, hence we can apply dominated convergence theorem to conclude
\[
\int_{-\pi \, \sqrt{n}}^{\pi \, \sqrt{n}} \left\vert \psi_n \bigg(\frac{x}{\sqrt{n}}\bigg) - e^{-\sigma^2 x^2/2} \right\vert dx = \int_{\R} \Psi_n (x) dx  \underset{n \rightarrow \infty}{\longrightarrow} 0.
\]
\end{proof}

\section{Proof of Proposition \ref{prop: Austin}}\label{Appendix C}
A probability preserving system $(X,\B,\mu,T)$ is \textbf{a Gaussian automorphism} if there exists $f\in L^2(X,\mu)$ such that
\begin{itemize}
\item $\int fd\mu=0$. 
\item The process $\left\{f\circ T^n\right\}_{n\in\Z}$ is Gaussian and the  
 $\sigma$-algebra generated by the functions $\left\{f\circ T^n\right\}_{n\in\Z}$ is $\B \mod \mu$. 
\end{itemize}
A Gaussian automorphism is spectrally determined and as a consequence for every $d\in\Z\setminus\{0\}$ it has a $d$'th root. That is there exists $S$ a probability preserving transformation of $(X,\B,\mu)$ such that $T=S^d$. In addition, when $T$ is an ergodic, zero-entropy, Gaussian automorphism then its $d$'th root is also an ergodic, zero-entropy, Gaussian automorphism,

\begin{proof}[Proof of \cref{prop: Austin}]
By the main theorem in \cite{2024arXiv240708630A}, there exists $S,T$ two ergodic, zero entropy Gaussian automorphisms of a probability space $(X,\B,\mu)$ and $f\in L^2(\mu)$ such that the averages 
\[
\mathbb{A}_n(f,f):=\sum_{k=0}^{n-1}f\circ T^nf\circ S^n
\]
do not converge in $L^2(\mu)$, As $S$ and $T$ are Gaussian automorphisms, there exists $R,Q$, two ergodic, zero-entropy, Gaussian automorphisms of $(X,\B,\mu)$ such that $R^c=T$ and $Q^d=S$. Since for all $n\in\N$,
\[
B_n(f,f):=\sum_{k=0}^{n-1}f\circ R^{cn} f\circ Q^{dn}=\sum_{k=0}^{n-1}f\circ T^nf\circ S^n,
\]
the averages $B_n(f,f)$ do not converge in $L^2(\mu)$.
\end{proof}

\bibliographystyle{alpha}
\bibliography{embedding}

\newcommand{\SortNoop}[1]{}
\begin{thebibliography}{{Ryz}24}

\bibitem[AR62]{Abramov_Rohlin_1972}
L.~M. Abramov and V.~A. Rohlin.
\newblock Entropy of a skew product of mappings with invariant measure.
\newblock {\em Vestnik Leningrad. Univ.}, 17(7):5--13, 1962.

\bibitem[{Aus}24]{2024arXiv240708630A}
Tim {Austin}.
\newblock {Non-convergence of some non-commuting double ergodic averages}.
\newblock {\em arXiv e-prints}, page arXiv:2407.08630, July 2024.

\bibitem[BD87]{MR891642}
Robert Burton and Manfred Denker.
\newblock On the central limit theorem for dynamical systems.
\newblock {\em Trans. Amer. Math. Soc.}, 302(2):715--726, 1987.

\bibitem[BL02]{Bergelson_Leibman_2002}
V.~Bergelson and A.~Leibman.
\newblock A nilpotent {R}oth theorem.
\newblock {\em Invent. Math.}, 147(2):429--470, 2002.

\bibitem[CL84]{MR788966}
Jean-Pierre Conze and Emmanuel Lesigne.
\newblock Th\'{e}or\`emes ergodiques pour des mesures diagonales.
\newblock {\em Bull. Soc. Math. France}, 112(2):143--175, 1984.

\bibitem[Con99]{MR1721618}
J.-P. Conze.
\newblock Sur un crit\`ere de r\'ecurrence en dimension 2 pour les marches stationnaires, applications.
\newblock {\em Ergodic Theory Dynam. Systems}, 19(5):1233--1245, 1999.

\bibitem[DGK21]{MR4290512}
George Deligiannidis, S\'ebastien Gou\"ezel, and Zemer Kosloff.
\newblock The boundary of the range of a random walk and the {F}\o lner property.
\newblock {\em Electron. J. Probab.}, 26:Paper No. 110, 39, 2021.

\bibitem[Dur10]{Durrett_2010}
Rick Durrett.
\newblock {\em Probability: theory and examples}, volume~31 of {\em Cambridge Series in Statistical and Probabilistic Mathematics}.
\newblock Cambridge University Press, Cambridge, fourth edition, 2010.

\bibitem[FH23]{MR4585298}
Nikos Frantzikinakis and Bernard Host.
\newblock Multiple recurrence and convergence without commutativity.
\newblock {\em J. Lond. Math. Soc. (2)}, 107(5):1635--1659, 2023.

\bibitem[HSY24a]{2024arXiv240710728H}
Wen {Huang}, Song {Shao}, and Xiangdong {Ye}.
\newblock {A counterexample on multiple convergence without commutativity}.
\newblock {\em arXiv e-prints}, page arXiv:2407.10728, July 2024.

\bibitem[HSY24b]{Huang_Shao_Ye_2024}
Wen Huang, Song Shao, and Xiangdong Ye.
\newblock A counterexample on polynomial multiple convergence without commutativity.
\newblock {\em Bull. Soc. Math. France}, 152(1):149--168, 2024.

\bibitem[Kec95]{Kechris_1995}
Alexander~S. Kechris.
\newblock {\em Classical descriptive set theory}, volume 156 of {\em Graduate Texts in Mathematics}.
\newblock Springer-Verlag, New York, 1995.

\bibitem[KV22]{MR4374685}
Zemer Kosloff and Dalibor Voln\'y.
\newblock Local limit theorem in deterministic systems.
\newblock {\em Ann. Inst. Henri Poincar\'{e} Probab. Stat.}, 58(1):548--566, 2022.

\bibitem[{Ryz}24]{2024arXiv240713741R}
Valery~V. {Ryzhikov}.
\newblock {Non-commuting transformations with non-converging 2-fold ergodic averages}.
\newblock {\em arXiv e-prints}, page arXiv:2407.13741, July 2024.

\bibitem[Sch98]{MR1663750}
Klaus Schmidt.
\newblock On joint recurrence.
\newblock {\em C. R. Acad. Sci. Paris S\'er. I Math.}, 327(9):837--842, 1998.

\bibitem[Vol99]{MR1624218}
Dalibor Voln\'{y}.
\newblock Invariance principles and {G}aussian approximation for strictly stationary processes.
\newblock {\em Trans. Amer. Math. Soc.}, 351(8):3351--3371, 1999.

\bibitem[Wal12]{MR2912715}
Miguel~N. Walsh.
\newblock Norm convergence of nilpotent ergodic averages.
\newblock {\em Ann. of Math. (2)}, 175(3):1667--1688, 2012.

\end{thebibliography}

\end{document}